\documentclass[11pt,reqno,a4paper]{amsart}

\usepackage{amssymb,epsfig,graphics}
\usepackage{amsmath}
\usepackage{hyperref}
\usepackage{color}

\newtheorem{theorem}{Theorem}[section]
\newtheorem{proposition}[theorem]{Proposition}

\newtheorem{corollary}[theorem]{Corollary}
\newtheorem{lemma}[theorem]{Lemma}
\newtheorem{sub-lemma}[theorem]{Sub-Lemma}

\newtheorem{remark}[theorem]{Remark}
\newtheorem{example}[theorem]{Example}

\newcommand{\bom}{b_{\omega_0}}
\def\A{\mathcal{A}}
\def\L{\mathcal{L}}

\def\G{\mathcal{G}}

\def\L{\mathcal{L}}

\def\P{\mathcal{P}}
\def\R{\mathcal{R}}
\def\T{\mathcal{T}}

\def\NN{\mathbb{N}}
\def\PP{\mathbb{P}}
\def\RR{\mathbb{R}}

\DeclareMathOperator{\sign}{sign}
\DeclareMathOperator{\Lip}{Lip}

\let\eps=\varepsilon

\def\B{\mathcal{B}}

\def\RR{{\mathbb R}}
\def\1{{{\mathit 1} \!\!\>\!\! I} }

\newcommand{\amin}{{\underline \alpha}_0}
\newcommand{\amax}{{\underline \alpha}_1}

\begin{document}

\title{Linear response for random dynamical systems}
\author{Wael Bahsoun}
\author{Marks Ruziboev}
\address{Department of Mathematical Sciences, Loughborough University,
Loughborough, Leicestershire, LE11 3TU, UK}
\email{W.Bahsoun@lboro.ac.uk, M.Ruziboev@lboro.ac.uk}
\author{Beno\^\i t Saussol}
\address{Universit\'e de Brest, Laboratoire de
Math\'ematiques de Bretagne Atlantique, CNRS UMR 6205, Brest, France}
\email{benoit.saussol@univ-brest.fr}
\thanks{This work was conducted during visits of WB and MR  to Universit\'e de Bretagne Occidentale and of BS to Loughborough University. WB and BS would like to thank The Leverhulme Trust for supporting their research visits through the Network Grant IN-2014-021. WB and MR would like to thank The Leverhulme Trust for supporting their research through the research grant RPG-2015-346.}
\begin{abstract}
We study for the first time linear response for random compositions of maps, chosen independently according to a distribution $\PP$. We are interested in the following question: how does an absolutely continuous stationary measure (acsm) of a random system change when $\PP$ changes smoothly to $\PP_{\eps}$? For a wide class of one dimensional random maps, we prove differentiability of acsm with respect to $\eps$; moreover, we obtain a linear response formula.  We apply our results to iid compositions, with respect to various distributions $\PP_{\eps}$, of uniformly expanding circle maps, Gauss-R\'enyi maps (random continued fractions) and Pomeau-Manneville maps. Our results yield an \emph{exact formula} for the invariant density of random continued fractions; while for Pomeau-Manneville maps our results provide a \emph{precise relation} between their linear response under certain random perturbations and their linear response under deterministic perturbations. 
\end{abstract}
\date{\today}
\maketitle
\markboth{Wael Bahsoun, Marks Ruziboev \and Beno\^\i t Saussol}{Linear response for random dynamical systems}
\bibliographystyle{plain}
\tableofcontents
\section{Introduction}
Existence and stability of absolutely continuous invariant measures\footnote{In this paper we focus on absolutely continuous invariant measures since they naturally fit with the systems we consider. In particular, for the systems we consider, absolutely continuous invariant measures are the so-called physical measures, the ones that provide information for a large set of initial conditions.} are main ingredients to study statistical properties of chaotic dynamical systems. In particular, a question that is interesting from both theoretical and applied point of views is how does an absolutely continuous invariant measure change, and consequently the statistical properties of the system, if the original system changes slightly? 

\bigskip

It is known that for certain perturbations of deterministic dynamical systems one can prove that the measure changes smoothly and obtain a formula for the derivative, called the \emph{linear response formula}. Linear response for deterministic dynamical systems has been pioneered\footnote{See \cite{KP} for earlier related work.} by Ruelle \cite{R} followed by Dolgopyat \cite{D} and Baladi \cite{Ba1} among others \cite{BS, BBS, BaS0, BaS, BT, BL, D, GL, K}. For numerical results on linear response see \cite{BGNN, PV}. Negative results, where linear response does not hold are also known \cite{Ba1, Ba2, BaS0}. For progress in this direction of research see the survey article \cite{Ba2} and the recent articles \cite{BKL, Ja}. 

\bigskip

However, to the best of our knowledge there are no results in the literature on linear response for random compositions of maps. Our goal in this paper is to pioneer this direction and to provide a new point of view for perturbations in the random setting. Indeed, in this work we study linear response for random compositions of maps, chosen independently according to a distribution $\PP$. We are interested in the following question: how does an absolutely continuous stationary measure (acsm) of a random system change when $\PP$ changes smoothly to $\PP_{\eps}$? For a wide class of one dimensional random maps, we prove differentiability of acsm with respect to $\eps$; moreover, we obtain a linear response formula.  We apply our results to iid compositions of uniformly expanding circle maps, to iid compositions of the Gauss-R\'enyi maps and to iid compositions of Pomeau-Manneville maps. The latter family models intermittent transition to turbulence and is of central interest for both mathematicians \cite{DT, FFTV, G, Hu, LS, LSV, Me, SvS, young99} and physicists \cite{PM}, while the former family provides fundamental links between ergodic theory and number theoretic questions \cite{DO, DdeV, KKV}. Indeed, for the Gauss-R\'enyi maps we use our results to approximate the invariant density governing the statistics of random continued fractions by the well known invariant density of the Gauss map, $\frac{1}{\log2}\frac{1}{1+x}$, and its linear response with respect to a Bernoulli distribution (see subsection \ref{appGauss} for more details; in particular \eqref{GR:revisted}). In the case of Pomeau-Manneville maps we show that the linear response with respect to a family of uniform distributions converging to a Dirac $\delta_{\alpha_0}$-distribution, $\alpha_0\in(0,1)$, amounts to \emph{half} of the linear response with respect to deterministic perturbations (see subsection \ref{PMdirac}).

\bigskip

The paper is organised as follows. In Section \ref{ss:hatT} we study iid compositions of piecewise uniformly expanding, piecewise $C^3$ and onto interval maps. Under suitable assumptions on $\PP_\eps$ we prove differentiability of the stationary density as an element of $C^1$. Our main result in this section is Theorem \ref{thm:LRuniform}. In Section \ref{nonuniRDS} we study random dynamical systems whose constituent maps are non-uniformly  expanding. For this purpose we introduce an inducing scheme and obtain an induced random dynamical system which satisfies the assumptions of Section \ref{ss:hatT}. We relate the stationary densities of the induced random system to the original one and prove differentiability of the stationary density of the original random system as an element of a weighted $C^0$-norm. Our main result in this section is Theorem \ref{thm:A}. In Section \ref{LRformula} we obtain linear response formulae for the systems studied in Sections \ref{ss:hatT} and \ref{nonuniRDS} when the map $\eps\mapsto\PP_\eps$ is a distribution of order one. Moreover, we provide several examples of natural families where $\eps\mapsto\PP_\eps$ is a distribution of order one. Section \ref{LSVmaps} contains several examples of families of maps and distributions that satisfy the conditions of Sections \ref{ss:hatT} and \ref{nonuniRDS} respectively. In particular, it contains examples that studies iid compositions of the Gauss-R\'enyi maps, an approximation of the invariant density of random Gauss-R\'enyi maps (see equation \eqref{GR:revisted}), random compositions of Pomeau-Manneville maps chosen in an iid fashion according to a family of smooth distributions $\PP_\eps$, and random compositions of Pomeau-Manneville maps chosen in an iid fashion according to a family of uniform distributions $\PP_\eps$ converging to a Dirac $\delta$-distribution. Section \ref{appendix} is an appendix which contains facts about distributions of order one, a proof of linear response for Markov operators with a uniform spectral gap, and a proof of a uniform spectral gap on $C^i$, $i=1,2$, for the transfer operators associated with the systems studied in Section \ref{ss:hatT}. 

\section{Piecewise uniformly expanding random dynamical systems}\label{ss:hatT}

In this section we introduce a class of (family of) random dynamical systems whose constituent maps are uniformly expanding,
with a finite or countable number of branches, for which we will be able to 
prove a linear response formula.

\subsection{A class of uniformly expanding maps}\label{class}

Let $X$ be a compact interval, and $m$ be the normalized Lebesgue measure on $X$. 
Let $(\Omega,\PP)$ be a probability space.
Let $T_\omega\colon X\to X$, $\omega\in\Omega$ be a family of maps such that for each $\omega\in\Omega$, there exists a finite or countable set $Z_\omega$ and a partition (mod 0) of $X$ into open intervals $X_{z,\omega}$, $z\in Z_\omega$ such that the restriction of $ T_\omega$ to $X_{z,\omega}$ is $C^3$ and onto. We denote by $g_{z,\omega}$ the inverse branches of $ T_\omega$ on $X_{z,\omega}$.
For convenience we take the same labelling set $Z$ for all the $\omega$'s. Since all the maps have finite or countable number of branches this is always possible by introducing empty branches $g_{z,\omega}$ whenever $z$ and $\omega$ are not compatible. In all the sums over $z$ that will appear we will not count these empty branches.


\subsection{Stationary measure of the Markov process}\label{markstat}
We study the random dynamical system defined by the i.i.d. composition of maps $T_\omega$,
with $\omega$ distributed according to $\PP$.
The random dynamical system induces a Markov process with
transition kernel 
\[
p(x,A) = \int_\Omega 1_A(T_\omega(x)) d\PP(\omega) .
\]
We say that a measure $\mu$ on $X$ is stationary if for any measurable $A\subset X$
\[
\int_X p(x,A)d\mu(x) = \mu(A),
\]
or equivalently, for any $\phi\colon X\to \RR$ measurable and bounded,
\[
\int_X\int_\Omega \phi\circ T_\omega(x) d\PP(\omega) d\mu(x) = \int \phi(x) d\mu(x).
\]
For $\phi \in L^{\infty}(X)$ and $\Phi\in L^1(X)$ we have
\[
\begin{aligned}
\int_X\int_\Omega\phi\circ T_\omega \Phi d\PP(\omega)dm
&=
\int_\Omega\int_X\phi\circ T_\omega \Phi dm d\PP(\omega)\\
&=
\int_\Omega \int_X \phi L_{T_\omega}\Phi dm d\PP(\omega)\\
&=
\int_X \phi \int_\Omega L_{T_\omega}\Phi d\PP(\omega)  dm, 
\end{aligned}
\]
where $L_{T_\omega}$ is the transfer operator associated with the map $T_\omega$, defined by
\[
L_{T_\omega}\Phi = \sum_{z\in Z} \Phi\circ g_{z,\omega}\cdot |g_{z,\omega}'|.
\]
We set
$$
L_{\PP}\Phi := \int_\Omega L_{T_\omega}\Phi d\PP(\omega). 
$$
In particular, any stationary measure $\mu$ absolutely continuous with respect to $m$, with density $h$, satisfies
\[
L_\PP h=h.
\]
$ L_\PP$ is called the transfer operator of this random dynamical system. 

\subsection{The perturbed random system}\label{pertunif}
Let $\PP_\epsilon$ be a family of probability measures on $\Omega$.
We are interested in studying the change in the statistical behaviour of the random system\footnote{
A common way to have a parameter dependent random system is also when the system consists of a fixed probability space $(\Omega,\PP)$ and a parametrized family of maps $T_{\omega,\eps}$, $\omega\in\Omega$, $\eps\in V$.

This situation can be represented in our framework, with the new probability space $\Omega\times V$ and the probability measure $\PP_\eps=\PP\otimes\delta_\eps$.} $(\Omega, \{T_\omega\}, \PP_\eps)$ as $\eps$ changes in a neighborhood $V$ of $0$. The transfer operator of the perturbed system is denoted by $L_{\PP_\eps}$. We assume:
\begin{itemize}
\item[(A1)] there exists $\tilde D>0$ such that 
\begin{equation}\label{eq:dist}
\left|\frac{g_{z, \omega}'(x)}{g_{z, \omega}'(y)}-1\right| \le  \tilde D|x-y|
\end{equation}
for any $x, y \in X$, $z\in Z$ and $\omega\in \Omega$. Moreover, there exists $M>0$ independent of $\eps$ such that for $i=2, 3$ we have
\begin{equation}\label{eq:avgdist}
\sup_{\eps\in V}\sum_{z\in Z}\sup_{x\in X}\int_{\Omega}|g^{(i)}_{z,\omega}|d\PP_\eps(\omega)\le M.
\end{equation}
\item[(A2)] There exists $\beta\in(0, 1)$ such that $\displaystyle{\sup_{\omega\in\Omega}\sup_{z\in Z}\sup_{x\in X}|g_{z,\omega}'(x)|\le \beta}$.
\end{itemize}

\begin{proposition}\label{lem:unifgap}
Under assumptions (A1)-(A2), for each $\eps\in V$ the operators $L_{\PP_\eps}$ has a uniform spectral gap on $C^1$ and $C^2$. In particular, the random dynamical system $(\Omega, \{T_\omega\}, \PP_\eps)$ admits a unique stationary density $h_\eps\in C^2$.
\end{proposition}
We postpone the proof of Proposition~\ref{lem:unifgap} to the appendix. 
\begin{remark}
Assumptions (A1) and (A2) are only needed to insure a uniform spectral gap of $L_\eps$ on $C^1$ and the uniqueness of the stationary density $h_\eps\in C^2$. It may be possible to prove such properties under a different set of conditions. Thus, to keep the exposition about linear response as general as possible, we do not assume conditions (A1)- (A2) below. Instead, in condition B, we assume that $L_{\PP_\eps}$ has a uniform spectral gap on $C^1$ and the existence of a unique of the stationary density $h_\eps\in C^2$.
\end{remark}

\vskip 0.5cm
\noindent{\bf Assumption B}
Assume that $L_{\PP_\eps}$ admits a uniform spectral gap on $C^1$. Moreover, assume that the random dynamical system $(\Omega, \{T_\omega\},\PP_\eps)$ admits a unique stationary density $h_\eps\in C^2$.\\ 

For each $z\in Z$, and $\Phi\in L^1(X)$, let 
\begin{equation}\label{psiz}
\psi_z(\eps,x)=\int_\Omega[\Phi\circ g_{z,\omega} |g_{z,\omega}'|](x) d\PP_\eps(\omega).
\end{equation}
\begin{itemize}
\item
For $\Phi=h_0\in C^2$ we assume that the partial derivatives $\partial_\eps \psi_z(\eps,x),$ $ \partial_x \psi_z(\eps,x),$ $\partial_x\partial_\eps \psi_z(\eps,x)$, $\partial_\eps \partial _x\psi_z(\eps,x)$ exist and jointly continuous in $(\eps, x)$ on $X\times V$. Hence, $\partial_\eps \partial _x\psi_z(\eps,x) = \partial _x\partial_\eps \psi_z(\eps,x)$. Moreover we assume that for $i=0,1$ we have
\begin{equation}\label{a3}
\sum_{z\in Z}\sup_{\eps\in V}  \sup_{x\in X}|
\partial_\eps \psi_z^{(i)}(\eps,x)| <\infty,
\end{equation}
 where $\psi_z^{(0)}=\psi_z$ and $\psi_z^{(1)}=\partial_x\psi_z$.
\item
In addition, we assume that for any $\Phi\in C^1$, $\psi_z(\eps,x)$ and $\partial_x \psi_z(\eps,x)$ exists and are jointly continuous. Moreover, for $i=0,1$ we assume that
\begin{equation}\label{a3'}
\sum_{z\in Z}\sup_{\eps\in V}\sup_{x\in X}|
 \psi_z^{(i)}(\eps,x)| <\infty. 
\end{equation}
\end{itemize}

\begin{theorem}\label{thm:LRuniform}
Let $(\Omega,\{T_\omega\},\PP_\eps)$ be a family of random dynamical systems as described above. 
Under assumption B, the density $ h_\eps$ of the stationary measure is differentiable as a $C^1$ element at $\eps=0$, 
that is there exists $ h^*\in C^1$ such that
\begin{equation}
\| \frac{h_\eps-h_0}{\eps}-h^* \|_{C^1} \to 0.
\end{equation}
In addition, the following linear response formula holds\footnote{Explicit linear response formulae will be derived in Section \ref{LRformula}.}:
\begin{equation}\label{eq:response} 
 h^*:=(I- L_{\PP_0})^{-1} \partial_\eps L_{\PP_\eps}h_0|_{\eps=0}, 
\end{equation}
where 
$$ 
\partial_\eps L_{\PP_\eps}h_0|_{\eps=0} = 
\partial_\eps \sum_{z\in Z}\int_\Omega[h_0\circ g_{z,\omega} |g_{z,\omega}'|] d\PP_\eps(\omega).
$$
\end{theorem}
\begin{proof}
Under assumption B, we verify in a series of lemmas below that the operators $L_{\PP_\eps}$ satisfy the assumptions\footnote{Proposition  \ref{pro:lrmarkov} provides general conditions to obtain linear response for systems whose Markov operators admit a uniform spectral gap on some Banach space.} of Proposition~\ref{pro:lrmarkov}.
\end{proof}

\begin{lemma} \label{diff-lemma}
The map $\eps\mapsto L_{\PP_\eps}h_0$ is differentiable at $\eps=0$ as a $C^1$ element.
\end{lemma}

\begin{proof}
We consider the  maps $\psi_z$ defined in Assumption B with $\Phi=h_0\in C^2$ (by Lemma~\ref{lem:unifgap}). 
Since $L_{\PP_\eps}h_0=\sum_z \psi_z(\eps)$ it suffices to show that

(i) for each $z\in Z$, the map $\eps\in V\mapsto \psi_z(\eps)\in C^1(X)$ is differentiable;

(ii) the series $\sum_{z\in Z} \sup_{\eps\in V} \|\partial_\eps\psi_z\|_{C^1(X)}<\infty$.

We only prove (i) since (ii) follows from \eqref{a3}.

By the commutation relations given by the first item of assumption B we have
\begin{equation}\label{eq:com}
\partial_\eps \psi_z(\eps)^{(i)} = (\partial_\eps \psi_z(\eps))^{(i)},\quad i=0,1
\end{equation}
and these are continuous functions on $X \times V$.\\

Let $v\in V$ and $\eps$ be small. We have
\begin{equation}\label{eq:W2norm}
\begin{split}
&\|\psi_z({\eps+v})-\psi_z(v)-\eps(\partial_\zeta \psi_z(\zeta)|_{\zeta=v})\|_{C^1(X)}=\\
&\hskip 3cm\sum_{i=0}^1
\|\psi_z^{(i)}({\eps+v})-\psi_z^{(i)}(v)-\eps(\partial_\zeta \psi_z^{(i)}(\zeta)|_{\zeta=v})\|_\infty.
\end{split}
\end{equation}
For each $x$, by the mean value theorem, there exists $t_{x,\eps}^i$ such that
$\psi_z^{(i)}({\eps+v},x)-\psi_z^{(i)}(v,x)=\eps \partial_\zeta  \psi_z^{(i)}(\zeta,x)|_{\zeta=t_{x,\eps}^i}$, with $|t_{x,\eps}^i-v|<\eps$. Therefore, by the joint continuity,
\[
\eqref{eq:W2norm}
\le |\eps| \sum_{i=0}^1\sup_{x\in X} | \partial_\zeta  \psi_z^{(i)}(\zeta,x)|_{\zeta=t_{x,\eps}^i}- \partial_\zeta \psi_z^{(i)}(\zeta,x)|_{\zeta=v} |
=o(\eps).
\]
\end{proof}

\begin{lemma}
For any $\phi\in C^1$, 
the map $\eps\mapsto L_{\PP_\eps}\phi$ is continuous at $\eps=0$ as a $C^1$ element.
\end{lemma}
\begin{proof}
We consider the  maps $\psi_z$ defined by \eqref{psiz} with a general $\Phi\in C^1$.  
Since $L_{\PP_\eps}\Phi=\sum_z\psi_z(\eps)$ it suffices to show that

(i) for each $z\in Z$, the map $\eps\in V\mapsto \psi_z(\eps)\in C^1(X)$ is continuous;

(ii) the series $\sum_{z\in Z} \sup_{\eps\in V} \|\psi_z\|_{C^1(X)}<\infty$.

We only prove (i) since (ii) follows from \eqref{a3'}.

Let $v\in V$ and $\eps$ be small. We have
\begin{equation}\label{eq:W3norm}
\|\psi_z({\eps+v})-\psi_z(v)\|_{C^1(X)}
=
\sum_{i=0}^1
\|\psi_z^{(i)}({\eps+v})-\psi_z^{(i)}(v)\|_\infty=o(\eps)
\end{equation}
by joint continuity.
\end{proof}

\section{Random dynamical systems with an inducing scheme}\label{nonuniRDS}

\subsection{Family of maps}\label{setup}

Let $X=[0,1]$ and $\Delta$ a closed subinterval of $X$.
We consider a finite number $\ell\ge1$ of one-parameter families of maps.
The parameter of the $k$th family is defined on a compact interval $I_k\subset\RR$ 
(we allow the possibility of $I_k$ to be reduced to a single point) $k=1,...,\ell$.
For each $k$ and $u\in I_k$ the map $T_{k,u}\colon X\to X$ is
piecewise $C^1$ with a finite or countable number of monotonic full branches.
Let $\P_{k,u}$ be the partition of monotonicity intervals of $T_{k,u}$.
We assume that $\Delta$ is a union of elements of $\P_{k,u}$ for all $k,$ and $u\in I_k$.
Let $N^{in}$ and $N^{out}$ be two disjoint copies of $\NN$. We enumerate the branches of $T_{k,u}$
starting from $\Delta$ with $N^{in}$ and the others with $N^{out}$.

Let $L_{T_{k,u}}$ denote the transfer operator associated with $T_{k,u}$, which is defined by 
\[
L_{T_{k,u}}(\phi)(x) = \sum_{T_{k,u}(y)=x} \phi(y) \frac{1}{|T_{k,u}'(y)|}.
\]
For convenience and the purpose of inducing, 
to keep track of the family $k$ we add it into the label of the branch.
Let $S^{in}= \{1,\ldots,\ell\}\times N^{in}$, $S^{out}= \{1,\ldots,\ell\}\times N^{out}$, and let $S=S^{in}\cup S^{out}$ be the new labelling set.

When $s$ and $k$ are compatible, that is the first coordinate of $s$ is $k$, 
we denote by $g_{s,k,u}$ the inverse of the $s$th branch of $T_{s,k,u}$.
Then $L_{T_{k,u}}$ reads
\begin{equation}\label{forF}
L_{T_{k,u}}(\phi) = \sum_{s\in S} \phi\circ g_{s,k,u}\cdot |g_{s,k,u}'|,
\end{equation}
where we ignore, here and in the rest of the paper, terms that correspond to non compatible $s$ and $k$.

\subsection{Stationary measure of the Markov process}\label{markstatoo}

Let $\Omega=\{1,\ldots,\ell\}\times \RR$. Let $\PP$ be a probability measure on $\Omega$,
supported on $\cup_{k=1}^\ell \{k\}\times I_k$. Let $\pi$ be the marginal measure of $\PP$ on $\{1,\ldots,\ell\}$ and 
$\eta_{k}$ be the conditional measure of $\PP$ on $\{k\}\times I_k$. We study the random dynamical systems defined by the i.i.d. composition of maps $T_\omega$,
with $\omega$ distributed according to $\PP$.
The random dynamical system induces a Markov process with
transition kernel 
\[
p(x,A) = \int_\Omega 1_A(T_\omega(x)) d\PP(\omega).
\]
We say that a measure $\mu$ on $X$ is stationary if for any measurable $A\subset X$
\[
\int_X p(x,A)d\mu(x) = \mu(A),
\]
or equivalently, for any $\phi\colon X\to \RR$ measurable and bounded,
\[
\int_X\int_\Omega \phi\circ T_\omega(x) d\PP(\omega) d\mu(x) = \int \phi(x) d\mu(x).
\]
For $\phi \in L^{\infty}(X)$ and $\psi\in L^1(X)$ we have
\[
\begin{aligned}
\int_X\int_\Omega\phi\circ T_\omega \psi  d\PP(\omega)dm
&=
\int_\Omega\int_X\phi\circ T_\omega \psi dm d\PP(\omega)\\
&=
\int_\Omega \int_X \phi L_{T_\omega}\psi dm d\PP(\omega)\\
&=
\int_X \phi \int_\Omega L_{T_\omega}\psi d\PP(\omega)  dm 
\end{aligned}
\]
We set
$$
L_{\PP}\psi := \int_\Omega L_{T_\omega}\psi d\PP(\omega). 
$$
In particular, any stationary measure $\mu$ absolutely continuous with respect to $m$, with density $h$, satisfies
\[
L_\PP h=h.
\]
$ L_\PP$ is called the transfer operator of this random dynamical system. 

\subsection{Examples of random systems}
To illustrate the construction in Subsection \ref{setup}, we provide the following two examples\footnote{Later in Section \ref{LSVmaps} we will show that the random systems in both examples admit linear response for suitable perturbations of $\PP$.}:
\begin{example}[Gauss-R\'enyi]\label{Ex:GR} 
Let $\G$ and $\R$ be respectively the Gauss and R\'enyi transformations on the unit interval.
Recall that $\G(x)=1/x \mod 1$ and $\R(x)=1/(1-x)\mod 1$. 
The random system consists of choosing
randomly the Gauss and the R\'enyi map, with respective probabilities $p$ and $1-p$, with $p\in[0,1]$.
We model this example with $\ell=2$, $I_1=I_2=\{0\}$, $T_{1,0}=\G$, $T_{2,0}=\R$,
$\pi_{1}=p$, $\pi_2=1-p$ which determines $\PP$. This random dynamical system is used to study random continued fractions \cite{KKV}.
\end{example}
\begin{example}[Pommeau-Manneville]\label{Ex:PM}
Let $[\alpha_0,\alpha_1]\subset (0,1)$. For $\alpha\in [\alpha_0,\alpha_1]$ a map
 $\T_\alpha$  is defined by: 
  $$\T_{\alpha}(x)=\begin{cases}
       x(1+2^{\alpha}x^{\alpha}) \quad x\in[0,\frac{1}{2}]\\
       2x-1 \quad x\in(\frac{1}{2},1]
       \end{cases}.$$
The random system consists of choosing randomly the parameter $\alpha$ with probability $\eta$,
supported on $[\alpha_0,\alpha_1]$.
We model this example with $\ell=1$, $I_1=[\alpha_0,\alpha_1]$, $T_{1,\alpha}=\T_\alpha$,
$\pi_{1}=1$, $\eta_{1}=\eta$ which determines $\PP$. 
\end{example}

\subsection{An inducing scheme and assumptions}\label{ss:T}

Given $\hat\omega\in \Omega^\NN$, we write $T_{\hat\omega}^n:=T_{\omega_{n-1}}\circ \cdots \circ T_{\omega_1}\circ T_{\omega_0}:X\to X$. For $\hat\omega\in\Omega^\NN$ we define $\hat T_{\hat\omega}$ as the first return map under the orbit of $T_{\hat\omega}^n$ to $\Delta$; i.e., for $x\in\Delta$
$$\hat T_{\hat\omega}(x)=T_{\hat\omega}^{R_{\hat\omega}(x)}(x),$$
where 
$$R_{\hat\omega}(x)=\inf\{n\ge 1:\, T_{\hat\omega}^{n}(x)\in\Delta\}.$$
Let $Z$ be the set of finite sequences of the form $z = z_0z_1\ldots z_n$, 
where $z_0\in S^{in}$ and $z_i\in S^{out}$ for $i=1,\ldots,n$ and $n \in\mathbb N$.  
We denote by $|z|=n+1$ the length of the word $z\in Z$. 
We set $g_{z,\omega}=g_{z_0,\omega_0}\circ g_{z_1,\omega_1}\circ\cdots\circ g_{z_n,\omega_n}$. Then for $x\in X$ we have $T_{\hat\omega}^{n+1}\circ g_{z,\hat\omega}(x)=x$. For each $\hat\omega$, the cylinder sets $g_{z,\hat\omega}(\Delta)$, form a partition of $\Delta$ (mod $0$).
Note that on $g_{z,\hat\omega}(\Delta)$ we have $R_{\hat\omega}(\cdot)=n+1$. We assume that the maps $( \Delta,\{ T_{\hat\omega} \}_{\hat\omega\in\Omega^\NN})$ satisfy the assumptions in Subsection~\ref{class}: piecewise $C^3$ and piecewise onto.

\subsection{An induced random dynamical system}

Let $\hat\Omega=\Omega^{\NN}$ and $\hat \PP=\PP^{\NN}$. We study the random dynamical system defined by the i.i.d. composition of maps $\hat T_{\hat \omega}$,
with $\hat\omega$ distributed according to $\hat\PP$. Following the framework of Subsection \ref{markstat}, for $\Phi\in L^1(\Delta)$, the transfer operator of this induced random system is given by 
\begin{equation}\label{LhatP}
\hat L_{\hat\PP} \Phi =\int_{\hat\Omega}\hat L_{\hat\omega}\Phi d\hat\PP(\hat\omega)= \sum_{z\in Z} \int_{\hat\Omega} \Phi\circ g_{z,\hat\omega}|g'_{z,\hat\omega}| d\hat\PP(\hat\omega),
\end{equation}
where $\hat L_{T_{\hat\omega}}$ is the transfer operator associated with $\hat T_{\hat\omega}$. Notice that $\hat L_{\hat\PP}$ reduces to  
$$\hat L_{\hat\PP} \Phi=\sum_{z\in Z} \int_{\Omega} \Phi\circ g_{z,\omega}|g'_{z,\omega}| d\PP(\omega).$$
The density $\hat h$ of any absolutely continuous stationary measure $\hat\mu$ of the induced random system satisfies
\[
\hat L_{\hat \PP} \hat h= \hat h.
\]
\subsection{Unfolding the density of the induced random dynamical system}
For $\Phi\in L^1(\Delta)$, let $F_{\PP}(\Phi)\colon X\to\RR$ be defined by
\begin{equation}\label{theF}
F_{\PP}(\Phi):=1_\Delta\Phi + (1-1_\Delta)\sum_{z\in Z}\int_{\hat\Omega}\Phi\circ g_{z,\hat\omega}|g_{z,\hat\omega}'|d\hat\PP(\hat\omega).
\end{equation}
Note that $F_{\PP}$ is a linear operator. In the next lemma we show that if $\hat h$ is a stationary density for the induced random system then $F_\PP\hat h$ is a stationary density (up to normalization\footnote{Working with un-normalized densities keeps the operator $F_\PP$ linear.})  for the original one discussed in Subsection \ref{markstatoo}.
\begin{proposition}\label{Le:unfold}
Let $\hat h\in L^{1}(\Delta)$ be such that $\hat L_{\hat\PP}\hat h=\hat h$. 
Then $L_{\PP}(F_\PP\hat h)=F_\PP\hat h$.
\end{proposition}
\begin{proof}
Using the expression \eqref{forF} for $L_\PP$ we get
\[
\begin{aligned}
& L_{\PP}F_\PP\hat h = 
\int_\Omega L_{T_\omega} F\hat h d\PP(\omega) \\
& =
\sum_{s\in S^{in}}
\int_\Omega (F\hat h)\circ g_{s,\omega}\cdot |g'_{s,\omega}| d\PP(\omega) 
+
\sum_{s\in S^{out}}
\int_\Omega (F\hat h)\circ g_{s,\omega}\cdot |g'_{s,\omega}| d\PP(\omega) 
\\
&=
\sum_{s\in S^{in}}
\int_\Omega \hat h\circ g_{s,\omega}\cdot |g'_{s,\omega}| d\PP(\omega) 
\\
&+
\sum_{s\in S^{out}}
\int_\Omega  \sum_{z\in Z}\left(\int_{\hat \Omega} \hat h\circ g_{z,\hat\omega}\cdot |g'_{z,\hat\omega}| d\hat\PP(\hat\omega)\right)\circ g_{s,\omega}\cdot |g'_{s,\omega}| d\PP(\omega) \\
&=(I) + (II).
\end{aligned}
\]
The expression (II) can be rewritten, 
\[
 \sum_{z\in Z}\sum_{s\in S^{out}}\int_{\Omega^{|z|+1}}
 \left( \hat h\circ g_{z,\omega_0\ldots\omega_{|z|-1}}\cdot |g'_{z,\omega_0\ldots\omega_{|z|-1}}| \right)\circ g_{s,\omega_{|z|}}\cdot |g'_{s,\omega_{|z|}}| d\PP^{|z|+1}(\omega_0\ldots\omega_{|z|})
\]
Therefore, by using the fact that for any $n\in\NN$, $(\omega_0,\ldots,\omega_n,\omega)$ (under $\hat\PP\times\PP$) has the same distribution as $\omega_0\cdots\omega_{n+1}$ under $\hat\PP$,
\[
(II) = \sum_{z\in Z,|z|\ge 2}\int_{\hat\Omega} \hat h\circ g_{z,\hat\omega}\cdot |g'_{z,\hat\omega}| d\hat\PP(\hat\omega).
\]
Finally, 
\[
(I)+(II)  = \sum_{z\in Z}\int_{\hat\Omega} \hat h\circ g_{z,\hat\omega}\cdot |g'_{z,\hat\omega}| d\hat\PP(\hat\omega).
\]
Therefore, on $\Delta$, $L_{\PP}F_\PP\hat h$ is equal, by \eqref{LhatP}, to $\hat L_{\hat\PP}\hat h=\hat h =F_\PP\hat h$,
and outside $\Delta$ it is equal to $F_\PP\hat h$ by definition.
\end{proof}

\subsection{The perturbed random system}\label{PRTS}
 Let $\PP_\eps$ be a family of probability measures on $\Omega$, supported on $\cup_{k=1}^\ell \{k\}\times I_k$. Let $V$ be a neighbourhood of $0$. 
Let $\pi_\eps$ be the marginal measure of $\PP_\eps$ on $\{1,\ldots,\ell\}$ and 
$\eta_{k,\eps}$ be the conditional measure of $\PP_\eps$ on $I_k\times\{k\}$. As in Subsection~\ref{pertunif},
we are interested in studying the change in the statistical behaviour of the random system $(\Omega, \{T_\omega\}, \PP_\eps)$ as $\eps$ changes.  We assume that $\hat L_{\hat\PP_\eps}$ admits a uniform spectral gap on $C^1$. Moreover, assume that the random dynamical system $(\hat\Omega,\{\hat T_{\hat\omega}\}, \hat\PP_\eps)$  admits a unique stationary density $\hat h_\eps\in C^2$. By Proposition \ref{Le:unfold}, the stationary densities of the original random system $(\Omega,\{T_{\hat\omega}\}, \PP_\eps)$ and the induced one $(\hat\Omega,\{\hat T_{\hat\omega}\}, \hat\PP_\eps)$ are related by\footnote{Note that $h_\eps$ is un-normalized. When $h_\eps$ is integrable; i.e., when the random system preserves a probabilistic acsm, once the derivative of $h_\eps$ is obtained, the derivative of the normalized density can be easily computed. Indeed, if $h_\eps=h +\eps h^*+o(\eps)$, then $\int h_\eps=\int h +\eps \int h^*+o(\eps)$. Thus, $\partial_{\eps}(\frac{h_\eps}{\int h_\eps}){|}_{\eps=0}=h^*-h\int h^*$.}
\begin{equation}\label{eq:density}
h_\eps = F_{\PP_\eps}(\hat h_\eps).
\end{equation} 
Let $\mathcal{H}$ denote the set of continuous functions on $(0,1]$ with the norm $$\parallel f\parallel_{\mathcal{H}}=\sup\limits_{x\in(0,1]}|x^{\gamma}f(x)|,$$ for a fixed $\gamma\ge0$. When equipped with the norm $\parallel \cdot\parallel_{\mathcal{H}}$, $\mathcal{H}$ is a Banach space. For each $z\in Z$, and $\Phi\in L^1(X)$, let 
\begin{equation}\label{hpsiz}
\hat\psi_z(\eps,x)=\int_{\hat\Omega}[\Phi\circ g_{z,\hat\omega} |g_{z,\hat\omega}'|](x) d\hat\PP_\eps(\hat\omega).
\end{equation}
In Theorem \ref{thm:A} below we prove the differentiability, at $\eps=0$, of $\eps\mapsto h_\eps$ as an element of $\mathcal H$ under the following set of conditions:
\begin{itemize}
\item
For $\Phi=\hat h_0\in C^2$ we assume that the partial derivatives $\partial_\eps \hat\psi_z(\eps,x)$, $\partial_x \hat\psi_z(\eps,x)$, $\partial_x\partial_\eps \hat\psi_z(\eps,x)$, $\partial_\eps \partial _x\hat\psi_z(\eps,x)$ exist and are jointly continuous on $\Delta\times V$, whence, satisfy the commutation relation  $\partial_\eps \partial _x\hat\psi_z(\eps,x)\\ =\partial _x\partial_\eps \hat\psi_z(\eps,x)$.   For any $\Phi\in C^1$, $\hat\psi_z(\eps,x)$ and $\partial_\eps \hat\psi_z(\eps,x)$  exist and are jointly continuous on $(0,1]\times V$.  
Moreover we assume that for $i=0,1$ we have
\begin{equation}\label{ha3}
\sum_{z\in Z}\sup_{\eps\in V}  \sup_{x\in\Delta}|
\partial_\eps \hat\psi_z^{(i)}(\eps,x)| <\infty,
\end{equation}
where $\hat\psi_z^{(0)}=\hat\psi_z$ and $\hat\psi_z^{(1)}=\partial_x\hat\psi_z$.
\item For $\Phi\in C^1$, $\psi_z(\eps,x)$ and $\partial_x \psi_z(\eps,x)$ exists and are jointly continuous. Moreover, for $i=0,1$ we assume that
\begin{equation}\label{ha3'}
\sum_{z\in Z}\sup_{\eps\in V}\sup_{x\in\Delta}|
 \hat\psi_z^{(i)}(\eps,x)| <\infty. 
\end{equation}
\item For any $\Phi\in C^0$, we assume that
\begin{equation}\label{ha4}
\sum_{z\in Z}\sup_{\eps\in V}\|
 \hat\psi_z(\eps,\cdot)\|_{\mathcal H} <\infty. 
\end{equation}
\item
For $\Phi=\hat h_0\in C^2$ we assume that 
\begin{equation}\label{ha4'}
\sum_{z\in Z}\sup_{\eps\in V}  \|
\partial_\eps \hat\psi_z(\eps,\cdot)\|_{\mathcal H} <\infty.
\end{equation}

\end{itemize}
\begin{theorem}\label{thm:A}
Let $(\Omega,\{T_\omega\},\PP_\eps)$ be a family of random dynamical systems defined as
 in Subsection~\ref{PRTS}. Then
\begin{enumerate}
\item there exists $h^*\in\mathcal H$ such that
$$\lim_{\eps\to 0}||\frac{h_{\eps}-h_0}{\eps}-h^*||_{\mathcal H}=0;$$
i.e., $h_\eps$ is differentiable as an element of $\mathcal H$ with respect to $\eps$ at $\eps=0$.
\item In particular, if the conditions \eqref{ha4} and \eqref{ha4'} hold for $\gamma<1$ in the definition of $\mathcal H$, then
$$\lim_{\eps\to 0}||\frac{h_{\eps}-h_0}{\eps}-h^*||_{1}=0.$$
\end{enumerate}
\end{theorem}
\begin{remark}
An explicit formula of the derivative $h^*$ is given in Section~\ref{LRformula}.
\end{remark}
\begin{proof}(of Theorem \ref{thm:A})
Assumptions \eqref{ha3} and \eqref{ha3'} state that the induced system satisfies the assumptions of Section~\ref{ss:hatT}. Let $ \hat h^* = \partial_\eps\hat h_\eps|_{\eps=0}$ be  given by Theorem~\ref{thm:LRuniform} 
applied to the induced system (with hat). The argument starts from the first order expansion of $\hat h_\eps$ in $C^1$ 
\[
\hat h_\eps = \hat h + \eps \hat h^* + o(\eps).
\]
Using this, we then obtain, by the second statement of Lemma~\ref{3rdstep} below and relation ~\eqref{eq:density} the following expansion in $\mathcal H$
\[
h_\eps =F_{\PP_\eps}(\hat h_\eps) = F_{\PP_\eps}(\hat h) + \eps F_{\PP_\eps}(\hat h^*) + o(\eps).
\]
Finally, we obtain by Lemma~\ref{4thstep} below and the first statement of Lemma~\ref{3rdstep} below the first order expansion of $h_\eps$ in $\mathcal H$
\[
h_\eps = h + \eps( Q\hat h + F_{\PP}(\hat h^*)) + o(\eps),
\]
which finishes the proof of the theorem. 
\end{proof}
\begin{lemma}\label{3rdstep}
$F_{\PP_\eps}(\hat h^*)\to F_{\PP}(\hat h^*)$ in $\mathcal H$ and $F_{\PP_\eps}$ is uniformly bounded in $\L(C^0,\mathcal H)$.
\end{lemma}
\begin{proof} 
To prove uniform boundedness we use assumption \eqref{ha4} to get, for $\Phi \in C^0$,
\begin{equation*}
\begin{split}
||F_{\PP_\eps}(\Phi)||_{\mathcal H}&=||1_\Delta\Phi + (1-1_\Delta)\sum_{z\in Z}\hat\psi_z(\eps,\cdot)||_{\mathcal H}\\
&\le ||\Phi||_{C^0}+\sum_{z\in Z}\sup_{\eps\in V}\|
 \hat\psi_z(\eps,\cdot)\|_{\mathcal H} <\infty.
\end{split}
\end{equation*}
Next, to show $F_{\PP_\eps}(\hat h^*)\to F_{\PP}(\hat h^*)$ in $\mathcal H$, it is sufficient to prove

(i) for each $z\in Z$, the map $\eps\in V\mapsto \hat\psi_z(\eps,\cdot)\in\mathcal H$, defined with $\Phi=\hat h^*$, is continuous;

(ii) the series $\sum_{z\in Z} \sup_{\eps\in V} \|\hat\psi_z(\eps,\cdot)\|_{\mathcal H}<\infty$.\\
Notice that (ii) is implied by condition \eqref{ha4}. Moreover, condition \eqref{ha4} implies that $\hat\psi_z(\eps,\cdot)\in\mathcal H$. Finally, For $\Phi=\hat h^*\in C^1$, the map $\hat\psi_z(\eps,\cdot)$ is jointly continuous on $(0,1]\times V$ by assumption. This implies (i).
\end{proof}

\begin{lemma}\label{4thstep}
The map $\eps\mapsto  F_{\PP_\eps}\hat h$ is differentiable as an element in $\mathcal H$ and 
$\partial_\eps F_{\PP_\eps}\hat h|_{\eps=0} = Q\hat h$, where $Q$ is defined by
\begin{equation}\label{eq:Q}
Q\Phi = (1-1_\Delta) \sum_{z\in Z} \partial_\eps\hat\psi_z(\eps,\cdot)|_{\eps=0}.
\end{equation}
for any differentiable function $\Phi$.
\end{lemma}
\begin{proof}
It suffices to show that

(i) for each $z\in Z$, the map $\eps\in V\mapsto \hat\psi_z(\eps,\cdot)\in\mathcal H$, defined with $\Phi=\hat h^*$ is differentiable;

(ii) the series $\sum_{z\in Z} \sup_{\eps\in V} \|\partial_\eps\hat\psi_z(\eps,\cdot)\|_{\mathcal H}<\infty$.\\
We only prove (i) since (ii) holds under assumption \eqref{ha4'}. For (i) let $v\in V$ and $\eps$ be small. For each $x$, by the mean value theorem, there exists $t_{x,\eps}$ such that
$\psi_z({\eps+v},x)-\psi_z(v,x)=\eps \partial_\zeta  \psi_z(\zeta,x)|_{\zeta=t_{x,\eps}}$, with $|t_{x,\eps}-v|<\eps$. Therefore,
\begin{equation}\label{eq:hW2norm}
\begin{split}
&\sup_{x\in(0,1]}|x^{\gamma}[\psi_z({\eps+v})-\psi_z(v)-\eps(\partial_\zeta \psi_z(\zeta)|_{\zeta=v})]|\\
&\hskip 2cm\le |\eps| \sup_{x\in (0,1]} | x^{\gamma}[\partial_\zeta  \psi_z(\zeta,x)|_{\zeta=t_{x,\eps}}- \partial_\zeta \psi_z(\zeta,x)|_{\zeta=v}] |
=o(\eps),
\end{split}
\end{equation}
where we have used joint continuity of $\partial_\zeta  \psi_z$ on $(0,1]\times V$.
\end{proof}
\section{Explicit computation of the linear response formula}\label{LRformula}
To obtain an explicit linear response formula we assume that for all $k$, 

(i) the maps $\eps\mapsto \pi_\eps(k)$ are $C^1$.

(ii) the map $\eps\mapsto\eta_{k,\eps}$ is $C^1$ as distribution of order one. That is, for any $C^1$ function $\varphi\colon I_k\to \RR$
\begin{equation}\label{eq:C1noise}
\partial_\eps \int_{I_k} \varphi d\eta_{k,\eps}(u) = \int_{I_k} \frac{d\varphi}{du} d\nu_{k,\eps}(u),
\end{equation}
where $\nu_{k,\eps}$, $\eps\in V$, is a continuous family of signed measures with bounded total variation. 
\subsection{Natural families of distributions} \label{natural}
Before obtaining an explicit formula of the linear response under assumption \eqref{eq:C1noise}, we first present families of distributions on an interval $I$ satisfying assumption \eqref{eq:C1noise}.
\begin{enumerate}
\item \label{a} The first and easiest class is the translations of Dirac measures. Let $a\in I$ be  an interior point and let $V$ be such that $a+\eps\in I$ for any $\eps \in V$.  Then $\eta_\eps=\delta_{\eps+a}$ is a smooth family of measures.  Indeed, for any $\eps, \eps_0\in V$ and  $\varphi\in C^1(I)$ we have 
$$\begin{aligned}
&\langle \varphi, \delta_{\eps+a}\rangle =\varphi(a+\eps)=\varphi(a+\eps_0) + (\eps-\eps_0)\frac{d\varphi}{du}|_{u=(a+\eps_0)}+o(|\eps-\eps_0|)
\\
&=\langle \eta_{\eps_0}, \varphi\rangle +(\eps-\eps_0)\langle \eta_{\eps_0}, \frac{d\varphi}{du} \rangle+o(|\eps-\eps_0|).
\end{aligned}
$$
\item The second class of examples is the convex combinations of Dirac measures: $\eta_\eps=\sum_{i=1}^N\rho_i(\eps)\delta_{a_i+\eps}$, 
where $\rho_i$, $i=1, \dots, N$ is a family of non-negative, smooth functions with $\sum_{i=1}^N\rho_i=1$, and $a_i, a_i+\eps\in I$ for every $i=1, \dots, N$ and $\eps\in V$.     
By linearity and Example \eqref{a} for every  $\eps, \eps_0\in V$ and  $\varphi\in C^1(I)$ we have 
$$\begin{aligned}
&\langle \varphi, \delta_{\eps+a}\rangle =\sum_{i=1}^N\rho_i\varphi(a_i+\eps)
\\&=\sum_{i=1}^N\rho_i\varphi(a_i+\eps_0) + (\eps-\eps_0)\sum_{i=1}^N\rho_i\frac{d\varphi}{du}|_{u=a+\eps_0}+o(|\eps-\eps_0|)
\\
&=\langle \eta_{\eps_0}, \varphi\rangle +(\eps-\eps_0)\langle \eta_{\eps_0},\frac{d\varphi}{du} \rangle+o(|\eps-\eps_0|).
\end{aligned}
$$
\item \label{c} Let $\mu$ be a finite Borel measure on $I$ and $\rho_\eps: I\to \mathbb R$  be a family of densities which is $C^1$ in $\eps$.  The family $d \eta_\eps= \rho_\eps d\mu$ is a smooth family of measures.  As above for any $\varphi\in C^1(I)$ we have 
$$\begin{aligned}
\langle \eta_\eps, \varphi \rangle =\int\varphi\rho_\eps d\mu =\int\varphi(\rho_{\eps_0}+(\eps-\eps_0)\partial_{\eps} \rho_\eps|_{\eps=\eps_0}+o(|\eps-\eps_0|))d\mu
\\ = \langle \eta_{\eps_0}, \varphi \rangle + (\eps-\eps_0)\langle \varphi, \tilde\nu_{\eps_0} \rangle +o(|\eps-\eps_0|),
\end{aligned}
$$
where $d\tilde\nu_{\eps_0}=\partial_{\eps} \rho_\eps|_{\eps=\eps_0}d\mu$. This shows differentiability of $\eta_\eps$ and the derivative is $\tilde\nu_\eps$. However, to show that $\eta_\eps$ satisfies \eqref{eq:C1noise}, letting $a$ be the left endpoint of $I$, we assume without lost of generality that $\varphi(a)=0$ (Indeed, one can define $\tilde\varphi(x)=\varphi(x)-\varphi(a)$ and notice that $\langle \eta_\eps, \varphi \rangle =\langle \eta_\eps, \tilde\varphi \rangle +\varphi(a)$). 
Then we have 
$$\begin{aligned}
\langle \eta_\eps, \varphi \rangle =\langle \eta_{\eps_0}, \varphi \rangle + (\eps-\eps_0)\int_I\int_a^u\varphi'(s)ds d\tilde\nu_{\eps_0}(u) +o(|\eps-\eps_0|)\\
=\langle \eta_{\eps_0}, \varphi \rangle + (\eps-\eps_0)\int_I\int_{I\cap ]s,+\infty[}\varphi'(s)d\tilde\nu_{\eps_0}ds +o(|\eps-\eps_0|) \\
=\langle \eta_{\eps_0}, \varphi \rangle + (\eps-\eps_0)\langle \varphi', \nu_{\eps_0} \rangle + o(|\eps-\eps_0|),
\end{aligned}
$$
where $\nu_{\eps_0}$ is the measure with density $\tilde\nu_{\eps_0}(I\cap ]s,+\infty[)$ with respect to Lebesgue on $I$.
\item The final class of examples we consider is a family of uniformly distributed measures which converges to a dirac measure. Let $a\in I$ be an interior point. For $\eps>0$, let $\eta_\eps=\rho_\eps du$ be a family of probability measures on $I$ with  
$$
\rho_\eps(u)=\begin{cases} \frac{1}{\eps} \quad \text{if} \quad u\in  (a, a+\eps),\\ 0 \quad \text{otherwise}, \end{cases}
$$ 
and let $\eta_0=\delta_a$.  By direct computation we have, at $\eps\neq0$, for any $\varphi\in C^1$
\[
\partial_\eps\frac1\eps\int_a^{a+\eps}\varphi(u) du = \frac{-1}{\eps^2}\int_a^{a+\eps}\varphi(u)du + \frac1\eps \varphi(a+\eps).
\]
We suppose without loss of generality that $\varphi(a)=0$ and proceed as in the previous example to get
\[
\partial_\eps \langle\eta_\eps,\varphi\rangle = 
\frac{-1}{\eps^2}\int_a^{a+\eps}\varphi'(u) (a+\eps-u)du+\frac1\eps \int_a^{a+\eps}\varphi'(u)du
=\langle \nu_\eps,\varphi'\rangle,
\]
where $\nu_\eps$ is the measure with density $\frac{u-a}{\eps^2}1_{(a,a+\eps)}$.
This shows differentiability at $\eps\neq0$. Furthermore, since $\nu_\eps$ converges weakly to $\frac12 \delta_a$
as $\eps\to0$, this implies the differentiability at $\eps=0$ as well, with $\nu_0=\frac12\delta_a$.
\end{enumerate}
\subsection{Explicit formulae}
Throughout this subsection we use the following notation:
\[
\begin{aligned}
a_{k,\eps} &:= \prod_{j=0}^{|z|-1} \pi_{k_j,\eps}\\
B_{k,\eps}(\Phi) &:= \int_{I_{k_0}\times\cdots\times I_{k_{|z|-1}}} \Phi\circ g_{z,u_0,\ldots,u_{|z|-1}} |g_{z,u_0,\ldots,u_{|z|-1}}'| 
d\eta_{k_0,\eps}\cdots d\eta_{k_{|z|-1},\eps}. \\
\end{aligned}
\]
We have
\[
\hat\psi_z(\eps,\cdot) 
= \int_{\hat\Omega} \Phi\circ g_{z,\hat\omega}\cdot |g'_{z,\hat\omega}| d\hat\PP_\eps(\hat\omega)
:= a_{k,\eps} B_{k,\eps}(\Phi)(\cdot).
\]
\begin{remark}
Note that when the random dynamical system is in the setting of Section \ref{ss:hatT}; i.e., when inducing is not required, $|z|=1$ in the definitions of $a_{k,\eps}$ and  $B_{k,\eps}$. Consequently, the linear response formulae derived below can be also adapted to the case of Section \ref{ss:hatT} with the appropriate simplification. 
\end{remark}
\begin{lemma}\label{tool}
Let $z\in Z$ and set $n+1=|z|$. For $k=(k_0,\ldots,k_n)$ compatible with $z$,
let $\PP_{z,j,\eps}=\eta_{k_0,\eps}\times \cdots\times \eta_{k_{j-1},\eps}\times \nu_{k_j,\eps}
\times \eta_{k_{j+1},\eps}\times \cdots\times \eta_{k_{n},\eps}$, where $\nu_{k_j,0}=\partial_\eps\eta_{k_j,\eps}$ in the sense of \eqref{eq:C1noise}.  
 The functions $a_{k,\eps}$ and $B_{k,\eps}(\hat h)$ are differentiable in $\eps$ and 
\begin{equation}\label{derivation}
\begin{aligned}
\partial_\eps a_{k,\eps}& = \sum_{j=0}^n \partial_\eps \pi_{k_j,\eps} \prod_{i\neq j} \pi_{k_i,\eps}\\
\partial_\eps B_{k,\eps} &=
\sum_{j=0}^n \int_{I_{k_0}\times\cdots\times I_{k_n}}  \left[
\hat h'\circ g_{z, u_0,\ldots,u_n} \partial_{u_j} g_{z, u_0,\ldots,u_n} | g_{z, u_0,\ldots,u_n}'| + \right.\\
&\left. + \hat h\circ g_{z, u_0,\ldots,u_n} \partial_{u_j}| g_{z, u_0,\ldots,u_n}'| \right] d\PP_{z,j, \eps}(u_0,\ldots,u_n)\\
\partial_\eps \hat\psi_z(\eps,\cdot)
&= \partial_\eps a_{k,\eps}\cdot  B_{k,\eps}(\hat h)(\cdot) +  a_{k,\eps}\cdot \partial_\eps B_{k,\eps}(\hat h)(\cdot).
\end{aligned}
\end{equation}
\end{lemma}
\begin{proof}
The fact that $\partial a_{k,\eps}$ is differentiable is obvious from the definition. The differentiability in $\eps$ of $B_{k,\eps}(\hat h)$ follows from Lemma~\ref{measure} and the rest is a direct calculation.
\end{proof}
\begin{corollary}\label{responseF}
Let $(\hat\Omega,\{\hat T_\omega\},\hat\PP_\eps)$ be a family of random dynamical systems defined on the interval $\Delta$ that satisfies the conditions of Section~\ref{pertunif}. Then the density $\hat h_\eps$ of the stationary measure is differentiable as a $C^1$ element. Moreover, under condition \eqref{eq:C1noise} we have the following explicit linear response formula
\begin{equation}\label{eq:response} 
\hat h^*:=(I-\hat L_{\hat\PP})^{-1} \partial_\eps\hat L_{\hat \PP_\eps}\hat h|_{\eps=0} ,
\end{equation}
where
\begin{equation}\label{explicit}
 \partial_\eps\hat L_{\hat \PP_\eps}\hat h|_{\eps=0}=\sum_{z}\left[ \partial_\eps a_{k,\eps}\cdot  B_{k,\eps}(\hat h) +  a_{k,\eps}\cdot \partial_\eps B_{k,\eps}(\hat h)\right]|_{\eps=0}.
\end{equation}
\end{corollary}
\begin{proof}
The differentiability of $\hat h_{\eps}$  and formula \eqref{eq:response} follow from Theorem \ref{thm:LRuniform}. The formula \eqref{explicit} follows from Lemma \ref{tool}. 
\end{proof}
\begin{remark}\label{detres1}
 Notice that when $\ell=1$ and $\eta_\eps:=\delta_{u+\eps}$, the explicit representation of formula \eqref{eq:response} is given by
$$\hat h^*:=(I-\hat L_u)^{-1}\hat L_u[A^1_u\hat h_u'+A^2_u\hat h_u],$$
where $\hat h_u'$ is the spatial derivative of $\hat h_u$, the invariant density of $\hat T_u$, and
$$A^1_u=-\left(\frac{\partial_{\eps}\hat T_{u+\eps}}{\hat T'_{u+\eps}}\right){\Big{|}}_{\eps=0},\hskip 0.5cm A^2_u= \left(\frac{\partial_{\eps}\hat T_{u+\eps}\cdot \hat T_{u+\eps}''}{\hat T_{u+\eps}'^2}-\frac{\partial_{\eps}\hat T_{u+\eps}'}{\hat T_{u+\eps}'}\right){\Big{|}}_{\eps=0},$$
which is the classical linear response formula for deterministic piecewise $C^3$, piecewise onto and uniformly expanding interval maps (See for instance \cite{Ba2, BS} for the deterministic case).  
\end{remark}

\begin{corollary}\label{detres2}
If $(\Omega, \{T_\omega\},\PP_{\eps})$ satisfies the assumptions of Section~\ref{PRTS},  then the density $ h_\eps$ of the stationary measure is differentiable as an $\mathcal H$ element. Moreover, under condition \eqref{eq:C1noise} we have the following explicit linear response formula
\begin{equation}\label{eq:lrf}
h^* = F_{\PP} (\hat h^*) + Q\hat h.
\end{equation}
where $\hat h^*$ is the response of the induced random system $(\hat\Omega, \{\hat T_{\hat \omega}\},\hat\PP_{\eps})$ and
\begin{equation}\label{explicit2}
 Q\hat h=\sum_{z}\left[ \partial_\eps a_{k,\eps}\cdot  B_{k,\eps}(\hat h) +  a_{k,\eps}\cdot \partial_\eps B_{k,\eps}(\hat h)\right]|_{\eps=0},
\end{equation}
where $ \partial_\eps a_{k,\eps}$ and $\partial_\eps B_{k,\eps}(\hat h)$ are given by the formulae in Lemma \ref{tool}.
\end{corollary}
\begin{proof}
The proof follows from Theorems \ref{thm:A} and Lemma \ref{tool}. 
\end{proof}

\section{Applications}\label{LSVmaps}
\subsection{Random uniformly expanding circle maps}
\begin{example} (expanding  circle maps) \label{ex:UnifExP_Rev}
Let 
$$T_1(x)=2x+\lambda \sin(2\pi x) \mod 1\text{ and }T_2(x)=2x \mod 1,$$ 
where $\lambda \in (-\frac{1}{2\pi}, \frac{1}{2\pi})$ is a fixed number. 
The random system consists of choosing randomly $T_1$ with probability $\eps$ and $T_2$ with probability $1-\eps$, which determine the $\PP_\eps$ on $\Omega=\{1,2\}$.
\end{example}
Notice that $T_1,T_2$ are smooth and uniformly expanding circle maps. Thus, inducing is not required and we can directly apply the results of Section \ref{ss:hatT} to see that the random system admits a stationary density $h_\eps$ and that this density is differentiable as a $C^1$ element at $\eps=0$. To obtain such a conclusion we check that assumptions (A1), (A2) and (B) are satisfied.\\ 

\noindent{\bf Verifying (A1) and (A2)}\\ 
For $i=2, 3$  we have 
\begin{equation}\label{A1for2xmod1}
\sum_{z\in Z}\int_{\Omega}|g^{(i)}_{z,\omega}|d\PP_\eps(\omega)
 =\eps\left(|g^{(i)}_{1, 1}|+ |g^{(i)}_{2, 1}|\right)+ (1-\eps)\left(|g^{(i)}_{1, 2}|+ |g^{(i)}_{2, 2}|\right) 
 \end{equation}
Notice that  $|g^{(i)}_{1, 2}|+ |g^{(i)}_{2, 2}| \le 1$ for $i=1, 2, 3$. Moreover, $T_1'(x)=2+2\pi\lambda\cos(2\pi x)$, $T_1''(x)=-4\pi^2\lambda\sin(2\pi x)$,   $T_1'''(x)=-8\pi^3\lambda\cos(2\pi x)$. In particular,  $T_1'(x)\ge 2-2\pi\lambda>1$. Thus, $\sup_x\frac{|T''x|}{(T'x)^2}\le D<\infty$. Consequently, (A1) is satisfied. Condition (A2) is satisfied with $\beta=(2-2\pi\lambda)^{-1}< 1$.\\

\noindent {\bf Verifying assumption B}\\
We have
$$
\psi_j(\eps, x)=\eps (\Phi\circ g_{j, 1}|g_{j, 1}'|)(x)+(1-\eps)(\Phi\circ g_{j, 2}|g_{j, 2}'|)(x), \quad \text{for } j=1, 2.
$$
Now, existence and continuity of partial derivatives of $\psi_j(\eps, x)$ is obvious. Since sums in \eqref{a3} and \eqref{a3'} reduces to a finite sum and the elements $\psi_j^{(i)}(\eps, x)$ are $|\partial_\eps\psi_j^{(i)}(\eps, x)|$ are uniformly bounded in $\eps$ and $x$, conditions \eqref{a3} and \eqref{a3'} are satisfied.
\begin{remark}
Using the above family, one can also verify that the following random dynamical system satisfies our conditions and admits linear response: $(\Omega,\{T_\omega\},\PP_\eps)$, where $\Omega:= (-\frac{1}{2\pi}, \frac{1}{2\pi})$, $T_\omega(x):=2x+\omega \sin(2\pi x) \mod 1$ and $\PP_{\eps}:=\pi+\eps\omega^3$. We leave this for the reader to verify. 
\end{remark}
\subsection{Random continued fractions: Gauss-R\'enyi maps}
\begin{example}[]\label{GR} 
Let $\G$ and $\R$ be respectively the Gauss and R\'enyi transformations on the unit interval.
Recall that $\G(x)=1/x \mod 1$ and $\R(x)=1/(1-x)\mod 1$. 
The random system consists of choosing
randomly the Gauss and the R\'enyi map, with respective probabilities $p_\eps$ and $1-p_{\eps}$, with $\lim_{|\eps|\to 0}p_\eps=p\in(0,1)$. We assume $\eps\mapsto p_\eps$ is $C^1$. Moreover, we assume $\exists\tilde\beta\in(0,1)$ such that $\max\{p_\eps, 1-p_\eps\}\le\tilde\beta$.
\end{example}
Notice that $\G,\R$ are smooth and piecewise onto.
It is worth noting that the individual maps are not uniformly expanding. However, the random system is expanding on average: $\sup_x\left(\frac{p_\eps}{\G'(x)}+\frac{1-p_\eps}{\R'(x)}\right)\le\max\{p_\eps, 1-p_\eps\}$.  Thus, inducing is not required. Indeed, one can use the assumption $\max\{p_\eps, 1-p_\eps\}\le\tilde\beta$ together with differentiating, in $x$, the transfer operator
\begin{equation}\label{GRopera}
L_\eps\Phi(x)=\sum_{n=1}^\infty \frac{p_{\eps}}{(n+x)^2}\Phi\left(\frac{1}{n+x}\right)+\frac{1-p_{\eps}}{(n+x)^2}\Phi\left(1-\frac{1}{n+x}\right)
\end{equation}
to obtain a uniform Lasota-Yorke inequality on $C^{i}$, $i=1,2$. The uniform spectral gap on both spaces will then follow from the fact that the system is random covering \cite{KKV}. Consequently, for each $\eps\in V$, the random system admits a unique stationary density $h_\eps\in C^2$. We now show that this density is differentiable as a $C^1$ element at $\eps=0$. To obtain such a conclusion we check that assumption (B) is satisfied. \\
 
 \noindent{\bf Verifying assumption (B)}\\
 Notice that
 $$\psi_z(\eps,x)=\frac{p_{\eps}}{(n+x)^2}\Phi\left(\frac{1}{n+x}\right)+\frac{1-p_{\eps}}{(n+x)^2}\Phi\left(1-\frac{1}{n+x}\right).$$
 Note that $\eps\mapsto p_\eps$ is $C^1$. Thus, assumption B is satisfied since, for $i=1,2$,
 \begin{equation}\label{BforGR}
\sum_{z\in Z}\int_{\Omega}|g^{(i)}_{z,\omega}|d\PP_\eps(\omega)
 =\sum_{n=1}^\infty \bigg|\left(\frac{1}{n+x}\right)^{(i)}\bigg| <\infty.
 \end{equation}
\subsection{Approximating the invariant density of random continued fractions}\label{appGauss} 
\begin{example}[] 
In this example, we revisit the Gauss and R\'enyi maps $\G(x)=1/x \mod 1$ and $\R(x)=1/(1-x)\mod 1$. The random system consists of choosing randomly the Gauss and the R\'enyi map with probabilities $1-\eps$ and $\eps$, respectively. 
\end{example}
Notice in this example $1-\eps$ is the weight on the Gauss map. The invariant density of this random map is not known explicitly \cite{KKV}. Using Theorem \ref{thm:LRuniform} we obtain that the invariant density, $h_\eps$, of this random map is approximated by the invariant density of the Gauss map and its linear response; i.e., 
\begin{equation}\label{GR:revisted}
h_\eps=h_{\G}+\eps(I-L_{\G})^{-1}\partial_\eps L_{\eps} h_{\G}|_{\eps=0}+o(\eps),
\end{equation}
where
$$
\partial_\eps L_{\eps} h_{\G}|_{\eps=0}=-h_{\G}+L_{\R}h_{\G},
$$
$h_{\G}=\frac{1}{\log2}\frac{1}{1+x}$ is the invariant density of the Gauss map,  $L_{\G}$, $L_{\R}$, $L_{\eps}$ are the transfer operators associated with $\G$, $\R$ and the random map respectively. Moreover, the error term, $o$, is in the $C^1$-topology.

To verify the assumptions of Theorem \ref{thm:LRuniform}  for this example,  we show below that the second iterate of the random map is (uniformly in $\eps$) expanding on average and consequently have a uniform spectral gap on $C^i$, $i=1,2$ and there is no need to induce in this case (see subsection \ref{SGGR} in the appendix for a proof). Moreover, the verification of assumption (B) follows verbatim as in the previous example. 
\subsection{Random Pomeau-Manneville maps and a family of smooth measures}\label{PMU}
 \begin{example}\label{PM} Let $[\alpha_0,\alpha_1]\subset (0,1)$. For $u\in [\alpha_0,\alpha_1]$ a map $T_{u}$ is defined by: 
  $$T_{u}(x)=\begin{cases}
       x(1+2^{u}x^{u}) \quad x\in[0,\frac{1}{2}],\\
       2x-1 \quad x\in(\frac{1}{2},1].
       \end{cases}$$
        \end{example}  
The above family of maps was popularized by the work of Liverani-Saussol-Vaineti \cite{LSV} which is a version of the famous Pomeau-Manneville family \cite{PM}. Throughout this section we assume $\alpha_1<2\alpha_0<\gamma\le1+\alpha_0$, where $\gamma$ is the constant in the definition of the $\mathcal H$ norm.  Note that this is only a constraint on the distance between $\alpha_1$ and $\alpha_0$ but not on their range; i.e., $\alpha_0$ can still be any value in $(0,1)$. We now verify assumptions (A1), (A2), \eqref{ha3}, \eqref{ha3'}, \eqref{ha4} and \eqref{ha4'} for the family of maps in Example \ref{PM}, with the family of probability measures $\eta_\eps$ on $[\alpha_0,\alpha_1]$ defined by $d\eta_\eps=\rho_\eps du$, where 
 \begin{equation}\label{sdensity}
 \rho_\eps=\frac{2}{(\alpha_1-\alpha_0)(\alpha_1+2\eps)}(u-\frac{\alpha_0}{2}+\eps),
 \end{equation}
with 
\begin{equation}\label{p-size}
|\eps|\le\frac{\alpha_0}{4}.
\end{equation}
The family of densities $\rho_\eps$ has the following property:
\begin{lemma}\label{lem:rhoeps}
$\dfrac{\partial_\eps\rho_\eps}{\rho_\eps}\le C$ for some $C$ independent of $\eps.$ 
\end{lemma}
\begin{proof}
Direct computation shows that 
$\frac{\partial_\eps\rho_\eps}{\rho_\eps}=\frac{2(\alpha_1+\alpha_0-2u)}{(\alpha_1+2\eps)(2u-\alpha_0+2\eps)}$. Since $|\eps|\le \alpha_0/4$ and $u\in[\alpha_0, \alpha_1]$ we have $(\alpha_1+2\eps)(2u-\alpha_0+2\eps)\ge (\alpha_1-\alpha_0/2)\alpha_0/2>\alpha_0^2/4$ and 
$|2(\alpha_1+\alpha_0-2u)|\le 2(\alpha_1-\alpha_0)$. Letting $C=8(\alpha_1-\alpha_0)/\alpha_0^2$ finishes the proof.
\end{proof}
Define the random system $(\Omega, \{T_\omega\},\PP_\eps)$. Set  $\Omega=[\alpha_0, \alpha_1]^{\mathbb N}$ and $\PP_\eps=\eta_\eps^{\mathbb N}$. In this example $g_{z,\omega}:=g_{z,u_0,u_1,\dots,u_{n}}=g_{0}\circ g_{u_1}\circ\cdots\circ g_{u_{n}},$ 
with $g_0(x):=\frac{x+1}{2}$ and $g_{u_i}:= T^{-1}_{u_i,1}$, where $T_{u_i,1}:=T_{u_i}|_{[0,1/2]}$, and $n=|z|-1$.  For each $\omega$ we define a  sequence of pre-images of $\frac{1}{2}$ as follows. 
Let  $x'_0(\omega)=1$, $x'_1(\omega)=\frac{3}{4}$,  and
\begin{equation}\label{eq:x_n}
x'_n(\omega)=g_{z,\omega}(\frac12) \,\, \text{for} \,\, n\ge 2.
\end{equation}  
The sequences $\{x'_n(\omega)\}$ will allow us to define the inducing procedure  for each $T_\omega$. Notice that the sequence $\{x'_n(\omega)\}_{n\ge 0}$ generates a partition $\P_{\omega}=\{(x'_n(\omega), x'_{n-1}(\omega)]\mid n\ge 0\}$ on $\left(\frac{1}{2}, 1\right]$. We define $\hat T_{\omega}$ as the first return map under the orbit of $T_{\omega}^n$ to $\Delta$; i.e., for $x\in\Delta$
$$\hat T_{\omega}(x)=T_{\omega}^{R_{\omega}(x)}(x),$$
where 
\begin{equation}\label{eq:lsv_return}
R_\omega|_{(x'_n(\omega), x'_{n-1}(\omega)]}=n.
\end{equation}
The random dynamical system $(\hat \Omega, \{\hat T_{\hat\omega}\},\hat\PP_\eps)$ is then defined with $\hat T_{\hat\omega}:=\hat T_{\omega}$, where $[\hat\omega]_0=\omega$.
Note that in this example ${\hat T}'_\omega(x)\ge 2$. Thus, (A2) is satisfied. To verify the (A1) we first introduce some  notation. 
Related to the random sequence $\{x_n'(\omega)\}$, we define another random sequence $\{x_n(\omega)\}$ which takes values in $[0,1/2]$. Let $x_1=\frac{1}{2}$ and $x_n(\omega)= \tilde g_{z, \omega}(\frac{1}{2})$, where $\tilde g_{z,\omega}:= g_{u_1}\circ\cdots\circ g_{u_{n}}$. Let $\amin, \amax\in \Omega$ be two constant  sequences whose entries are  $\alpha_0$ and $\alpha_1$ respectively. 
For all $n\ge 1$ and $\omega\in\Omega$ the following inequality holds   
\begin{equation}\label{changeq}
x_n(\amin)\le x_n(\omega) \le  x_n(\amax).
\end{equation}
The proof of the inequality is analogous to that of  Lemma 4.4 in \cite{BBD}. Moreover, it is well known  that 
$x_n(\amin) \sim 
\frac{1}{2} \alpha_0^{-{1}/{\alpha_0}}n^{-{1}/{\alpha_0}}$ so if we define
$c_n(\amin) := x_n(\amin)n^{{1}/{\alpha_0}}$  then 
$\lim_n c_n(\amin) = \frac{1}{2} \alpha_0^{-{1}/{\alpha_0}}:= c(\amin)$.
We define $c_n(\amax)$ and $c_n(\amax)$ analogously.
Therefore, \eqref{changeq} implies that 
\begin{equation}\label{eq:coarsest}
c_n(\amin)n^{-1/\alpha_0} \le x_n(\omega) \le  c_n(\amax)n^{-1/\alpha_1}.
\end{equation}
Note that $\frac{x_n(\sigma\omega)+1}{2}=x'_n(\omega)$, where $\sigma:\Omega\to\Omega$ is the one sided shift map.\\
 
\noindent{\bf Verifying (A1)}\\
 A key step to verify (A1) is the estimation of $E_{\eta_\eps}[x'_{n-1}(\omega)-x'_n(\omega)]$. This will be achieved by using the other random sequence $\{x_n(\omega)\}$. We first start with an auxiliary lemma and a corollary.
\begin{lemma}\label{normalization} 
Let $c\ge 1$ and $\tilde c(\eta_\eps)= \frac{\alpha_0+2\eps}{(\alpha_1-\alpha_0)(\alpha_1+2\eps)}$. Then, as $t \rightarrow \infty$
\begin{equation}
E_{\eta_\eps} \left[e^{-(cu-\alpha_0)t}\right]\sim
\tilde c(\eta_\eps)\cdot\frac{1}{ct}\cdot e^{-(c-1)\alpha_0t}. 
\end{equation}
\end{lemma}
\begin{proof}
We have 
\begin{align*}
&E_{\eta_\eps} \left[e^{-(cu-\alpha_0)t}\right]=
\frac{2e^{\alpha_0t}}{(\alpha_1-\alpha_0)(\alpha_1+2\eps)} \int_{\alpha_0}^{\alpha_1}e^{-cut}(u-\frac{\alpha_0}{2}+\eps)du\\
&\frac{2e^{-(c-1)\alpha_0t}}{(\alpha_1-\alpha_0)(\alpha_1+2\eps)}\cdot\frac{1}{ct}\left((\frac{\alpha_0}{2}+\eps+\frac{1}{ct})-e^{-c(\alpha_1-\alpha_0)t}(\alpha_1-\frac{\alpha_0}{2}+\eps+\frac{1}{ct}) \right)\\
&\sim \frac{1}{ct}\cdot\frac{\alpha_0+2\eps}{(\alpha_1-\alpha_0)(\alpha_1+2\eps)}e^{-(c-1)\alpha_0t}.
\end{align*}
\end{proof}
Lemma \ref{normalization} shows that assumption (5.11) of \cite{BBR} is satisfied\footnote{Indeed, we have
\begin{equation}
\begin{split}
&\frac{\log n}{n[x_{n}(\omega)]^{\alpha_0}} \ge \alpha_0 2^{\alpha_0} \frac{(\log n)}{n}\biggl\{\sum_{k=2}^n
\left[\frac{2c_k(\alpha_0)}{k^{\frac{1}{\alpha_0}}}\right]^{\omega_{n-k} - \alpha_0}\\
&\hskip 2 cm - \frac{1+ \alpha_0}{2}  \left[\frac{2c_k(\alpha_1)}{k^{\frac{1}{\alpha_1}}}\right]^{2\omega_{n-k} - \alpha_0}\biggr\}:= \frac{\log n}{n}\sum_{k=1}^n  X_k(\omega).
\end{split}
\end{equation}
Lemma \ref{normalization} shows that $\frac{\log n}{n}\sum_{k=1}^nE_{\eta_\eps}(X_k(\omega))\to \alpha_02^{\alpha_0}\tilde c(\eta_\eps)$. Corollary \ref{roldstuff} will then follow using large deviation estimates for independent random variables, see \cite{BBR} for details.}. Consequently, we obtain an upper bound on the $x_n(\omega)$. 
\begin{corollary}\label{roldstuff}
There exists $c>0$ independent of $\eps$ such that $x_n(\omega)\le 2c^{-1/\alpha_0}n^{-\frac{1}{\alpha_0}}(\log n)^{\frac{1}{\alpha_0}}$. In addition one can find constants $C>0$, $u>0$, $v\in (0, 1)$ independent\footnote{By \eqref{p-size} both $c$ and $C$ can be chosen independent of $\eps$.} of $\eps$ such that $\PP_\eps\{n_1>n\} \le Ce^{-un^v}$.
\end{corollary}
\begin{lemma}\label{Ilength}
There exists a $\hat C>0$, independent of $\eps$, such that
$$2E_{\eta_\eps}[x'_{n-1}(\omega)-x'_n(\omega)]=E_{\eta_\eps} [x_{n-1}(\sigma\omega)-x_{n}(\sigma\omega)]\le \hat C\frac{[\log n]^{\frac{\alpha_0+1}{\alpha_0}}}{n^{\frac{1}{\alpha_0}+1}}.$$
\end{lemma}
\begin{proof}
By definition $E_{\eta_\eps}[x'_{n-1}(\omega)-x'_n(\omega)]=\frac12E_{\eta_\eps} [x_{n-1}(\sigma\omega)-x_{n}(\sigma\omega)]$. Thus, it is enough to deal with $E_{\eta_\eps} [x_{n-1}(\sigma\omega)-x_{n}(\sigma\omega)]$. We have
\begin{equation}\label{step1}
\begin{split}
E_{\eta_\eps} [x_{n-1}(\sigma\omega)-x_{n}(\sigma\omega)]&= E_{\eta_\eps} [x_{n-1}(\sigma\omega)-x_{n}(\omega)]\\
&=E_{\eta_\eps} [2^{\omega_0} (x_n(\omega))^{\omega_0+1}],
\end{split}
\end{equation}
where we have used stationarity to write $E_{\eta_\eps} [x_{n}(\omega)]=E_{\eta_\eps} [x_{n}(\sigma\omega)]$ and that $T_{\omega_0} (x_n(\omega))=x_{n-1}(\sigma\omega)$ and $T_{\omega_0}(x_n(\omega))= x_n(\omega)+2^{\omega_0} (x_n(\omega))^{\omega_0+1}$. Using \eqref{step1}, Corollary \ref{roldstuff} and the fact that $0\le 2x_n(\omega)\le 1$, we obtain
\begin{equation}\label{step2}
\begin{split}
&E_{\eta_\eps} [x_{n-1}(\sigma\omega)-x_{n}(\sigma\omega)]\le E_{\eta_\eps} [2^{\alpha_0} (x_n(\omega))^{\alpha_0+1}]\\
&\le 2^{\alpha_0} \left(E_{\eta_\eps} [\chi_{\{n_1(\omega)\le n\}}\cdot x_n(\omega)^{\alpha_0+1}]+E_{\eta_\eps} [\chi_{\{n_1(\omega)> n\}}\cdot x_n(\omega)^{\alpha_0+1}]\right)\\
&=2^{2\alpha_0+1}c^{-1-\frac{1}{\alpha_0}}
\frac{[\log n]^{\frac{\alpha_0+1}{\alpha_0}}}{n^{\frac{1}{\alpha_0}+1}}+Ce^{-un^v}\le \hat C \frac{[\log n]^{\frac{\alpha_0+1}{\alpha_0}}}{n^{\frac{1}{\alpha_0}+1}}.
\end{split}
\end{equation}
\end{proof}
The following lemma is proved in \cite{BBD} (see  Lemma 4.8 in \cite{BBD}) using the Koebe principle \cite{MS}.
\begin{lemma}\label{lem:weakdist}
There exists $\tilde D>0$ such that 
\begin{equation}\label{eq:dist}
\left|\frac{g_{z, \omega}'(x)}{g_{z, \omega}'(y)}-1\right| \le  \tilde D|x-y|
\end{equation}
for any $x, y \in X$ and $z\in Z$ and any $\omega\in \Omega$. 
In particular, there exists a $D>0$ independent of $\omega$ (hence independent, of $\eps$), such that 
\begin{equation}\label{wbdd}
\frac{g_{z,  \omega}'(x)}{g_{z,  \omega}'(y)} \le D  \text{ for any }z\in Z \text{ and } x, y\in \Delta. 
\end{equation}
\end{lemma}
For $y\in \Delta$ let  $y_0=y$, and
\begin{equation}\label{eq:ynomega}
y_n(\omega)=g_{u_1}\circ \dots \circ g_{u_n}(y) \text{ for } n \ge 1.
\end{equation}
Then $x_n(\omega)\le y_n(\omega)\le x_{n-1}(\omega)$. Moreover, $g_{z, \omega}^\prime(y)=\frac{y_n'(\omega)}{2}$  and $y_0^\prime =1$, $y_0''=0$. For $\omega\in \Omega$ letting $w_{k}$, $k\ge 0$ be the elements of $\omega$ we have the following 
\begin{lemma}\label{lem:yiprime}
For any $n\ge 1$ 
$$
|y_{n}''(\omega)|\le\alpha_1 
(\alpha_1+1)2^{\alpha_1}y_n'(\omega)\sum_{j=1}^n y_j(\sigma^{n-j}\omega)^{\omega_{n-j}-1}|y_j'(\sigma^{n-j}\omega)|
$$
and
$$
\begin{aligned}
|y_{n}'''(\omega)|\le 
3\alpha_1(\alpha_1+1)2^{\alpha_1}y_n'(\omega)\sum_{j=1}^n y_j(\sigma^{n-j}\omega)^{\omega_{n-j}-1}|y_j''(\sigma^{n-j}\omega)|\\
+(1-\alpha_0^2)\alpha_12^{\alpha_1}{y_{n}'(\omega)}\sum_{j=1}^n y_j(\sigma^{n-j}\omega)^{\omega_{n-j}-2}(y_j'(\sigma^{n-j}\omega))^2.
\end{aligned}
$$
\end{lemma}
\begin{proof}
Since 
$y_{n-1}(\sigma\omega)=y_n(\omega)(1+(2y_n(\omega))^{\omega_0})$ by taking consequent derivatives of both sides we have 
$$
y_{n-1}'(\sigma\omega)=y_n'(\omega)\left[1+(1+\omega_0)(2y_n(\omega))^{\omega_0}\right],
$$
$$
\begin{aligned}
y_{n-1}''(\sigma\omega)=y_n''(\omega)\left[1+(1+\omega_0)(2y_n(\omega))^{\omega_0}\right]\\
+(y_n'(\omega))^2\omega_0(1+\omega_0)2^{\omega_0}y_n(\omega)^{\omega_0-1},
\end{aligned}$$
$$
\begin{aligned}
y_{n-1}'''(\sigma\omega)=y_n'''(\omega)\left(1+(1+\omega_0(2y_n(\omega))^{\omega_0}\right)\\
+3y_n''(\omega)y_n'(\omega)\omega_0(1+\omega_0)2^{\omega_0}y_n(\omega)^{\omega_0-1}\\
-(y_n'(\omega))^3\omega_0(\omega_0^2-1)2^{\omega_0}y_n(\omega)^{\omega_0-2},
\end{aligned}
$$
which imply 
\begin{equation}\label{eq:y2prime}
\frac{y_{n-1}''(\sigma\omega)}{y_{n-1}'(\sigma\omega)} =\frac{y_{n}''(\omega)}{y_{n}'(\omega)}+ y_n'(\omega)\frac{\omega_0(1+\omega_0)2^{\omega_0}y_n(\omega)^{\omega_0-1}}{1+(1+\omega_0)(2y_n(\omega))^{\omega_0}}
\end{equation}
and 
\begin{equation}\label{y3prime}
\begin{aligned}
\frac{y_{n-1}'''(\sigma\omega)}{y_n'(\sigma\omega)}=\frac{y_n'''(\omega)}{y_n'(\omega)}
+3y_n''(\omega)y_n(\omega)^{\omega_0-1}
\frac{\omega_0(1+\omega_0)2^{\omega_0}}{\left(1+(1+\omega_0)(2y_n(\omega))^{\omega_0}\right)}\\
-(y_n'(\omega))^2y_n(\omega)^{\omega_0-2}+
\frac{\omega_0(\omega_0^2-1)2^{\omega_0}}{\left(1+(1+\omega_0)(2y_n(\omega))^{\omega_0}\right)}.
\end{aligned}
\end{equation}
Using \eqref{eq:y2prime} and \eqref{y3prime}, for $i=2, 3$ we obtain:
\begin{equation}\label{telescope}
-\frac{y_n^{(i)}(\omega)}{y_n'(\omega)}=\sum_{j=1}^n\frac{y_{j-1}^{(i)}(\sigma^{n-j+1}\omega)}{y_{j-1}'(\sigma^{n-j+1}\omega)}-\frac{y_j^{(i)}(\sigma^{n-j}\omega)}{y_j'(\sigma^{n-j}\omega)}.
\end{equation}
Using \eqref{telescope} and the fact that for any $j$ we have $\alpha_0\le \omega_j\le \alpha_1$ finishes the proof. 
\end{proof}
\begin{lemma}\label{forA1}
\text{ }
\begin{enumerate}
\item There exists $C>0$ such that 
$$
\int_{\Omega}|g_{z, \omega}^{'}|d\PP_\eps(\omega)\le C n^{-1-\frac{1}{\alpha_0}}[\log n]^{\frac{\alpha_0+1}{\alpha_0}};
$$ 
$$
\int_{\Omega}|g_{z, \omega}^{''}|d\PP_\eps(\omega)\le C n^{-1-\frac{1}{\alpha_1}}[\log n]^{\frac{\alpha_0+1}{\alpha_0}}.
$$ 
\item Moreover, there exists $M>0$ such that 
$$\displaystyle{\sup_{\eps\in V}\sum_{z\in Z}\sup_{x\in X}\int_{\Omega}|g_{z, \omega}^{(i)}|d\PP_\eps(\omega)\le M}, \text{ for } i=1, 2, 3.$$
\end{enumerate}
\end{lemma}
\begin{proof}
By the Mean Value Theorem, there exists $\xi\in \Delta$ such that 
$
\frac 1 2 g_{z, \omega}'(\xi)= x_{n-1}'(\omega)- x_{n}'(\omega).
$ Therefore, by Lemma  \ref{lem:weakdist} for $x\in \Delta$ we have 
\begin{equation}\label{eq:mvt}
g_{z, \omega}'(x)\le 2 D(x_{n-1}'(\omega)- x_{n}'(\omega)).
\end{equation}
Thus, by Lemma \ref{Ilength}
\begin{equation}\label{derleg}
\int_{\Omega}|g_{z, \omega}'|d\PP_\eps \le 2D E_{\eta_\eps}[x'_{n-1}(\omega)-x'_n(\omega)] \le C [\log n]^{\frac{\alpha_0+1}{\alpha_0}}{n^{-\frac{1}{\alpha_0}-1}}.
\end{equation}

Again by the Mean Value Theorem and Lemma \ref{lem:weakdist} and \eqref{eq:coarsest} $y_j'(\sigma^{n-j}\omega)\le Dx_{j}(\sigma^{n-j}\omega)\le Cj^{-1/\alpha_1}$. Also by \eqref{eq:coarsest} we have 
$y_j(\sigma^{n-j}\omega)^{\omega_{n-j}-1}\le C j^{\frac{1-\alpha_0}{\alpha_0}}$. Substituting this into  
the first item of Lemma \ref{lem:yiprime} implies  
$$
|y_n''(\omega)|\le Cy_n'(\omega)\sum_{j=1}^n j^{-\frac{1}{\alpha_1}+\frac{1}{\alpha_0}-1} \le C y_n'(\omega)n^{\frac{1}{\alpha_0}-\frac{1}{\alpha_1}}.
$$
Therefore, by Corollary \ref{roldstuff} and \eqref{derleg}, we have
$$
\begin{aligned}
&\int_{\Omega}|y_n''(\omega)|d\PP_\eps(\omega) \le C\int_{\Omega} y_n'(\omega)n^{\frac{1}{\alpha_0}-\frac{1}{\alpha_1}}d\PP_\eps(\omega)= C n^{\frac{1}{\alpha_0}-\frac{1}{\alpha_1}}\int_{\Omega} y_n'(\omega)d\PP_\eps(\omega)\\
&\le C n^{\frac{1}{\alpha_0}-\frac{1}{\alpha_1}}[\log n]^{\frac{\alpha_0+1}{\alpha_0}}{n^{-\frac{1}{\alpha_0}-1}}=C n^{-1-\frac{1}{\alpha_1}}[\log n]^{\frac{\alpha_0+1}{\alpha_0}}.
\end{aligned}
$$
Similarly, using the second item of Lemma \ref{lem:weakdist} and \eqref{eq:coarsest}  we have 
$$
\begin{aligned}
|y_n'''(\omega)|
&\le C_1y_n'(\omega)\sum_{j=1}^n j^{\frac{1}{\alpha_0}-1}|y_n''(\sigma^{n-j}\omega)| \\
&+C_2y_n'(\omega)\sum_{j=1}^n j^{\frac{2}{\alpha_0}-1}\cdot j^{-\frac{2}{\alpha_1}} \\
\le C_1n^{-\frac{1}{\alpha_1}}&\sum_{j=1}^n j^{\frac{1}{\alpha_0}-1}|y_j''(\sigma^{n-j}\omega)| 
+C_2y_n'(\omega)\sum_{j=1}^n j^{\frac{2}{\alpha_0}-\frac{2}{\alpha_1}-1}\\
\le C_1n^{-\frac{1}{\alpha_1}}&\sum_{j=1}^n j^{\frac{1}{\alpha_0}-1}|y_j''(\sigma^{n-j}\omega)| 
+C_2y_n'(\omega)n^{\frac{2}{\alpha_0}-\frac{2}{\alpha_1}}.
\end{aligned}
$$
Finally, 
$$
\begin{aligned}
&\int_\Omega|y_n'''(\omega)|d\PP_\eps(\omega)\\
&\le C_1n^{-\frac{1}{\alpha_1}}\sum_{j=1}^n j^{\frac{1}{\alpha_0}-1}\int_\Omega|y_n''(\sigma^{n-j}\omega)|d\PP_\eps(\omega) 
+C_2n^{\frac{2}{\alpha_0}-\frac{2}{\alpha_1}}\int_\Omega y_n'(\omega)d\PP_\eps(\omega)\\ 
&\le C_1n^{-\frac{1}{\alpha_1}}\sum_{j=1}^n j^{\frac{1}{\alpha_0}-1}\cdot j^{-1-\frac{1}{\alpha_1}}[\log n]^{\frac{\alpha_0+1}{\alpha_0}} +C_2n^{-1+\frac{1}{\alpha_0}-\frac{2}{\alpha_1}}[\log n]^{\frac{\alpha_0+1}{\alpha_0}}.
\end{aligned}
$$
This finishes the proof since $2\alpha_0 >\alpha_1$.
\end{proof}

\noindent{\bf Verifying assumptions \eqref{ha3} and \eqref{ha3'}}. 
Note that the smoothness of $\Phi$  and $\hat\PP_\eps$ together with Lebesgue differentiation theorem imply the existence and continuity of the derivatives $\partial_{\eps}\hat\psi_z(\eps, x),$ $\partial_x\hat\psi_z(\eps, x),$ $\partial_x\partial_{\eps}\hat\psi_z(\eps, x),$ $\partial_\eps\partial_x\hat\psi_z(\eps, x).$ 

\begin{lemma}
For any  $\Phi\in C^2(\Delta)$ and $i=0, 1$ 
$$\sum_{z\in Z}\sup_{\eps\in V}  \sup_{x\in\Delta}|
\partial_\eps \hat\psi_z^{(i)}(\eps,x)| <\infty.$$
\end{lemma}
\begin{proof}
For any $\eps \in V$ by the definition of $\hat\psi_z(\eps, \Phi)$, the regularity of $\PP_\eps$ and Lebesgue differentiation theorem we have\footnote{Note that the family of measures $\eta_\eps$ in this example belongs to family (3) of Subsection \ref{natural}.}   
\begin{equation*}
\partial_\eps\hat\psi_z(\eps,\Phi)(x)=\sum_{j=1}^{n}\int_{I^n}[\Phi\circ g_{z,\hat\omega} |g_{z,\hat\omega}'|](x) \Pi_{i\not=j}^n\rho_{\eps}(\omega_i)\partial_\eps\rho_{\eps}(\omega_j) d(\bar\omega_n).
\end{equation*}
Thus, by Lemma \ref{lem:rhoeps} for any $x\in \Delta$
\begin{equation}
\begin{split}\label{pdeps}
|\partial_\eps\hat\psi_z(\eps,\Phi)(x)|\le & nC{\|\Phi\|}_{\infty}\sup_{x\in(1/2, 1]}\int_{\hat\Omega} |g_{z,\hat\omega}'(x)| d\hat\PP_\eps(\hat\omega)\\ 
&\le C n^{-\frac{1}{\alpha_0}}[\log n]^{\frac{\alpha_0+1}{\alpha_0}},
\end{split}
\end{equation}
where in the last inequality we have used the first item of Lemma \ref{forA1}.
Since $|z|=n+1$, summing over $n$ in \eqref{pdeps} completes the proof for $i=0$. 
For $i=1$  again by definition of $\hat\psi_z(\eps, \Phi)$  and Lebesgue differentiation theorem we have
\begin{align*}
& |\partial_\eps\hat\psi_z'(\eps,\Phi)(x)| \\
 &\le \sum_{j=1}^{n}{\|\Phi\|}_{C^1} \int_{I^n}(|g_{z,\hat\omega}'|^2+|g_{z,\hat\omega}''|)(x) \prod_{i\not=j}^n\rho_{\eps}(\omega_i)\partial_\eps\rho_{\eps}(\omega_j) d(\bar\omega_n).
\end{align*}
Hence, by Lemma \ref{lem:rhoeps}, we have 
\begin{align*}
|\partial_\eps\hat\psi_z'(\eps,\Phi)(x)| \le n {\|\Phi\|}_{C^1}\sup_{x\in(1/2, 1]}\int_{\Omega}(|g_{z,\hat\omega}'|^2+|g_{z,\hat\omega}''|)(x) d\PP_\eps(\omega) \\
\le Cn^{-\frac{1}{\alpha_1}}[\log n]^{\frac{\alpha_0+1}{\alpha_0}},
\end{align*}
where in the final estimate we have used the first item of Lemma \ref{forA1}. Since $|z|=n+1$ summing over $n$ finishes the proof.
\end{proof}

\begin{lemma}
For any  $\Phi\in C^1(\Delta)$ and $i=0, 1$ 
$$\sum_{z\in Z}\sup_{\eps\in V}  \sup_{x\in\Delta}|
 \hat\psi_z^{(i)}(\eps,x)| <\infty.$$
\end{lemma}
\begin{proof} For $i=0$ the proof follows by direct application of the second item of  \ref{forA1}. For $i=1$ by definition of  $ \hat\psi_z(\eps,x)$,  Lebesgue differentiation theorem, the first and the second items of  \ref{forA1} we have 
$$
| \hat\psi_z'(\eps,x)| \le {\|\Phi\|}_{C^1} \int_\Omega (|g_{z,\hat\omega}'|^2+|g_{z,\hat\omega}''|)(x) d\PP_\eps(\omega) \le C n^{-1-\frac{1}{\alpha_1}}[\log n]^{\frac{\alpha_0+1}{\alpha_0}}
$$
Summing over $n$ finishes the proof. 
\end{proof}

\noindent {\bf Verifying assumptions \eqref{ha4} and \eqref{ha4'}}.
Below without lost of generality suppose that $g_{z, \omega}$ has range $(x_{n}'(\omega), x_{n-1}'(\omega)]$. This in particular implies that $|z|=n$.
We first start with two technical lemmas
\begin{lemma}\label{1plusalpha}
For any $\alpha\in (0, 1)$, $\gamma\in (0, 1+\alpha_0]$ and $x>0$  the following holds 
$$
(1+(2x)^\alpha)^\gamma\le 1+\gamma(2x)^\alpha+\frac{\gamma^2}{2}(2x)^{2\alpha}
$$ 
\end{lemma}

\begin{proof}
We let $\phi_1(x)=(1+(2x)^\alpha)^\gamma$ and $\phi_2(x)=1+\gamma(2x)^\alpha+\frac{\gamma^2}{2}(2x)^{2\alpha}$.
Since $\phi_1(0)=\phi_2(0)=1$, it suffices to prove that $\phi_1'(x) <\phi_2'(x)$ for $x>0$.  Direct computation implies
$$\phi_1'(x)=\gamma\alpha 2^\alpha x^{\alpha-1}(1+(2x)^{\alpha})^{\gamma-1},$$ 
$$\phi_2(x)'=\alpha\gamma2^\alpha x^{\alpha-1}(1+\gamma2^{\alpha}x^{\alpha}).$$
For $\gamma\in (0, 1]$ we have $(1+(2x)^{\alpha})^{\gamma-1}\le 1 <1+\gamma2^{\alpha}x^{\alpha}$.
In the case $\gamma\in (1, 1+\alpha_0]$ by standard argument we have $(1+(2x)^{\alpha})^{\gamma-1}<1+(\gamma-1)(2x)^{\alpha} <1+\gamma2^{\alpha}x^{\alpha}$ which finishes the proof. 
\end{proof}

Let $\bom:=\frac{(T_{1, \omega}(x)/x)^{\gamma}}{T_{1, \omega}'(x)}:=\frac{(1+(2x)^{\omega_0})^{\gamma}}{1+(\omega_0+1)(2x)^{\omega_0}}$. 
Then by Lemma \ref{1plusalpha} we have 
\begin{equation}\label{ineq:bom}
\bom \le \frac{1+\gamma(2x)^{\omega_0}+\frac{\gamma^2}{2}(2x)^{2\omega_0}}{1+(\omega_0+1)(2x)^{\omega_0}} \le 1+ \frac{\gamma^2}{2}(2x)^{2\omega_0}.
\end{equation}
\begin{lemma}\label{le:tech}
There exists a constant $\mathcal D>0$  such that for any $y\in [0, \frac12]$ 
$$y^\gamma g'_{n, \omega}(y) \le \mathcal D {[x_{n-1}(\omega)]}^{\gamma}.$$
\end{lemma}
\begin{proof} 
Letting $T_{1, \sigma\omega}^{-n+1}y=T_{1, \omega_1 }^{-1}\circ \dots \circ T_{1, \omega_{n-1}}^{-1}y$ by definition of $g_{n, \omega}$ we have 
$$
y^\gamma g'_{n, \omega}(y)= \frac{y^\gamma}{(T_{1, \omega_{n-1}}\circ \dots T_{1, \omega_1}\circ T_{2, \omega_0})'(T_{2, \omega_0}^{-1}\circ T_{1, \sigma\omega}^{-n+1}y)}
$$
$$
=\frac{y^\gamma}{T_{1, \omega_{n-1}}'(T^{-1}_{1, \omega_{n-1}}y)\circ T_{1, \omega_{n-2}}'(T^{-2}_{1, \sigma^{n-2}\omega}y)\dots T_{1, \omega_{1}}'(T^{-n+1}_{1, \sigma\omega}y)T_{2, \omega_0}(T_{2, \omega_0}^{-1}\circ T_{1, \sigma\omega}^{-n+1}y)} 
$$
$$
=\frac{(T^{-n+1}_{1, \omega}(y))^\gamma}{2}\cdot\prod_{j=1}^{n-1} \frac{(T^{1-j}_{1, \omega}(y)/T^{-j}_{1, \omega}y)^\gamma}{T_{1, \omega_{n-j}}(T^{-j}_{1, \omega}y)}
 $$
 $$
\text{by Lemma \ref{1plusalpha}} \quad
\le \frac{(T^{-n+1}_{1, \omega}(x))^\gamma}{2}\cdot\prod_{j=1}^{n-1} \left(1+\frac{\gamma^2}{2}(2T^{-j}_{1, \omega}y)^{2\omega_0}\right)
$$
$$
\text{by inequality \eqref{eq:ynomega}}\quad =\frac{(T^{-n+1}_{1, \omega}(y))^\gamma}{2}\cdot \exp\left(\sum_{j=1}^{n-1} \log(1+\frac{\gamma^2}{2}(2x_j(\omega))^{2\omega_0}\right)
$$
$$
\le x_{n-1}(\omega)^\gamma \exp\left(\sum_{j=1}^{n-1}\frac{\gamma^2}{2}2^{\alpha_1}j^{-2\omega_0/\alpha_1}\right) \le \mathcal D [x_{n-1}(\omega)]^\gamma.
$$
\end{proof}
\begin{lemma}\label{le:v3}
For any $\Phi\in C^0$, we have
$$
\sum_{z\in Z}\sup_{\eps\in V}\|
 \hat\psi_z(\eps,\cdot)\|_{\mathcal H} <\infty. 
$$
\end{lemma}
\begin{proof}
Recall that 
$$
\hat\psi_z(\eps,x)=\int_{\hat\Omega}[\Phi\circ g_{z,\hat\omega} |g_{z,\hat\omega}'|](x) d\hat\PP_\eps(\hat\omega).
$$
Therefore, by Lemma \ref{le:tech}, Corollary \ref{roldstuff} and Lemma \ref{Ilength}, we have
\begin{equation}\label{sumitup1}
\begin{split}
&||\hat\psi_z(\eps,\Phi)\|_{\mathcal H}\le||\Phi||_\infty\left[\sup_{x\in(0,\frac12]}\int_{\hat\Omega} x^\gamma g_{z,\hat\omega}'(x) d\hat\PP_\eps(\hat\omega)+\sup_{x\in(\frac12,1]}\int_{\hat\Omega} g_{z,\hat\omega}'(x) d\hat\PP_\eps(\hat\omega)\right]\\
&\le C||\Phi||_\infty\left[\sup_{x\in(0,\frac12]}\int_{\hat\Omega}  [x_{n-1}(\omega)]^\gamma\PP_\eps(\hat\omega)+ \frac{[\log n]^{\frac{\alpha_0+1}{\alpha_0}}}{n^{\frac{1}{\alpha_0}+1}}\right]\\
 &\le C||\Phi||_\infty\sup_{x\in(0,\frac12]}\left[E_{\PP_\eps} [\chi_{\{n_1(\omega)\le n\}}\cdot x_n(\omega)^{\gamma}]+E_{\PP_\eps} [\chi_{\{n_1(\omega)> n\}}\cdot x_n(\omega)^{\gamma}]\right]\\
&\hskip 7cm +C||\Phi||_\infty +\frac{[\log n]^{\frac{\alpha_0+1}{\alpha_0}}}{n^{\frac{1}{\alpha_0}+1}}\\
 &\le C||\Phi||_\infty \left( E_{\PP_\eps} [\chi_{\{n_1(\omega)\le n\}}\cdot (n^{-1}\log n)^{\gamma/\alpha_0}]+ \PP_\eps \{n_1(\omega)\ge n\}\right)\\
&\hskip 7cm+C||\Phi||_{\infty} \frac{[\log n]^{\frac{\alpha_0+1}{\alpha_0}}}{n^{\frac{1}{\alpha_0}+1}}\\
&\le C||\Phi||_\infty  \hat C n^{-\frac{\gamma}{\alpha_0}}(\log n)^{\frac{\gamma}{\alpha_0}}+C||\Phi||_{\infty}\frac{[\log n]^{\frac{\alpha_0+1}{\alpha_0}}}{n^{\frac{1}{\alpha_0}+1}}.\\
\end{split}
\end{equation}
Since $|z|=n+1$ and $\gamma>2\alpha_0$, summing over $n$ in \eqref{sumitup1} completes the proof of the lemma.
\end{proof}
\begin{lemma}\label{le:v4} We have
$$
\sum_{z\in Z}\sup_{\eps\in V}  \|
\partial_\eps \hat\psi_z(\eps,\hat h_0)\|_{\mathcal H} <\infty.
$$
\end{lemma}
\begin{proof}
We first notice that
\begin{equation}\label{peps}
\partial_\eps\hat\psi_z(\eps,\hat h_0)(x)=\sum_{j=1}^{n}\int_{I^n}[\hat h_0\circ g_{z,\hat\omega} |g_{z,\hat\omega}'|](x) \Pi_{i\not=j}^n\rho_{\eps}(\omega_i)\partial_\eps\rho_{\eps}(\omega_j) d(\bar\omega_n).
\end{equation}
Since $\frac{\partial \rho_\eps}{\rho_\eps}\le C$ (independent of $\eps$, see Lemma \ref{lem:rhoeps}), using \eqref{peps} and an argument similar to that in the proof of Lemma \ref{le:v3}, we have
\begin{equation}\label{sumitup2}
\begin{split}
\|\partial_\eps \hat\psi_z(\eps,\hat h_0)\|_{\mathcal H}&\le C \sum_{j=1}^{n}||\int_{I^n}[\hat h_0\circ g_{z,\hat\omega} |g_{z,\hat\omega}'|]d\eta_{\eps}^n||_{\mathcal H}\\
&\le n C||\hat h_0||_\infty\left[\sup_{x\in(0,\frac12]}\int_{I^n}  [x_{n-1}(\omega)]^\gamma d\eta_{\eps}^n+\frac{[\log n]^{\frac{\alpha_0+1}{\alpha_0}}}{n^{\frac{1}{\alpha_0}+1}}\right] \\
&\le C||\hat h_0||_\infty \left[ n^{-\frac{\gamma}{\alpha_0}+1}+\frac{[\log n]^{\frac{\alpha_0+1}{\alpha_0}}}{n^{\frac{1}{\alpha_0}}}\right]. 
\end{split}
\end{equation}
Since $|z|=n+1$ and $\gamma>2\alpha_0$, summing over $n$ in \eqref{sumitup2} completes the proof of the lemma.
\end{proof}


\subsection{Random Pomeau-Manneville maps and a family of a uniformly distributed measures converging to a Dirac}\label{PMdirac}
Throughout this section we let $u\in(\alpha_0,\alpha_{0}+\eps)$, where $[\alpha_0,\alpha_0+\eps]\subset (0,1)$. We consider again the Pomeau-Manneville family defined in Example \ref{PM} but this time with $u$ distributed according to $\eta_\eps=\rho_\eps du$ where, for $\eps>0$, 
\begin{equation}\label{distB}
\rho_\eps(u)=\begin{cases} \frac{1}{\eps} \quad \text{if} \quad u\in  (\alpha_0,\alpha_{0}+\eps),\\ 0 \quad \text{otherwise}; \end{cases}
\end{equation}
and $\eta_0=\delta_{\alpha_0}$. Let $\PP=\eta_0^{\NN}$. The random dynamical system $(\hat \Omega, \{\hat T_{\hat\omega}\},\hat\PP_\eps)$ is defined in the same way as in the previous example of subsection \ref{PMU}, with the only difference is that $\hat\PP_\eps$ is defined using the density in \eqref{distB}. Assumptions (A1) and (A2) are verified following a similar approach to the proofs in  Lemmas \ref{lem:weakdist}, \ref{lem:yiprime} and \ref{forA1}. Verifying \eqref{ha3}, \eqref{ha3'}, \eqref{ha4} and \eqref{ha4'} for this type of distribution is done by using duality; i.e., for $i=0,1$:
\begin{equation*}
\begin{split}
&\partial_\eps\hat\psi_z^{(i)}(\eps,\Phi)(x)=\\
&\sum_{j=1}^{n}\int_{[\alpha_0,\alpha_0+\eps]^n}\frac{\partial}{\partial\omega_j}[\Phi\circ g_{z,\omega} |g_{z,\omega}'|]^{(i)}(x)d\eta_{\eps}(\omega_1)\cdots d\nu_\eps(\omega_j)\cdots d\eta_{\eps}(\omega_n),
\end{split}
\end{equation*}
with $\frac{d\nu_\eps}{du}=\frac{u-\alpha_0}{\eps^2}1_{(\alpha_0,\alpha_0+\eps)}$. Note that as $\eps\to 0$
$$\int_{[\alpha_0,\alpha_0+\eps]^n}\frac{\partial}{\partial\omega_j}[\Phi\circ g_{z,\omega} |g_{z,\omega}'|](x)d\eta_{\eps}(\omega_1)\cdots d\nu_\eps(\omega_j)\cdots d\eta_{\eps}(\omega_n)$$
converges to
$$\frac12\frac{\partial}{\partial\omega_j}[\Phi\circ g_{z,\omega} |g_{z,\omega}'|](x)|_{\omega_i=\alpha_0};\,\, i=1,\dots, n.$$
Therefore, the linear response formula in this example is given by:
$$\frac12 \left[F_{\PP} (\hat h^*_{\alpha_0}) + Q\hat h_{\alpha_0}\right],$$
which is $\frac12$ of the response under deterministic perturbations of $T_{\alpha_0}$ (see Corollary \ref{detres2} and Remark \ref{detres1}). This might be an interesting observation in relation to controlling the statistical properties of chaotic dynamical systems \cite{GP}. Here $h_{\alpha_0}$ is the invariant density of $T_{\alpha_0}$ and the explicit formula of $\hat h^*_{\alpha_0}$ is the same as in Remark \ref{detres1} with $\alpha_0$ replacing $u$. 
\section{Appendix}\label{appendix} 
\subsection{Elementary facts about convergence of measures}

We collect some elementary facts relative to convergence of measures.

Let $\mu_\eps$ be a family of finite measures on a compact rectangle $K$ of $\RR^d$. 
We say that $\mu_\eps$ is differentiable as a distribution of order one, if
there exists a finite vector valued measure $\nu_\eps$ of bounded total variation such that 
for any test function $\phi\colon K\to \RR$, 
\begin{equation}\label{C1dist}
\partial_\eps \int_K \phi(u) d\mu_\eps(u) = \int_K \nabla_u \phi\cdot d\nu_\eps(u).
\end{equation}
We set $\partial_\eps\mu_\eps = \nu_\eps$.
\begin{lemma}\label{measure}
Let $K_i$ be a compact rectangles in $\RR^{d_i}$, $i=1,2$.

%

(i)
Let $\mu_\eps$ be a family of measures on $K_1$, differentiable as distributions of order one. 
Then $\mu_\eps$ is a continuous family of measures on $K_1$.

(ii) 
The product of two families of measures on $K_1$, $K_2$ respectively, 
differentiable as distributions of order one, is a distribution of order one  on $K_1\times K_2$, and 
\[
\partial_\eps (\mu_\eps^1\otimes\mu_\eps^2)=\partial_\eps \mu_\eps^1\otimes\mu_\eps^2+ \mu_\eps^1\otimes\partial_\eps\mu_\eps^2.
\]
\end{lemma}
\begin{proof}
(i) Let $\mu_\eps$ be a family of measures satisfying \ref{C1dist}. Therefore, by definition,
$$\int \phi d\mu_\eps=\int\phi d\mu_0+\eps\int \nabla_u \phi\cdot d\nu_0(u)+o(\eps).$$
Consequently, 
$$|\int \phi d\mu_\eps-\int\phi d\mu_0|\le\eps||\phi||_{C^1}||\nu_0||+o(\eps).$$
(ii) $$\int_{K_1} \phi(u_1,u_2) d\mu^1_\eps=\int_{K_1}\phi(u_1,u_2) d\mu^1_0+\eps\int_{K_1} \nabla_{u_1} \phi(u_1,u_2)\cdot d\nu^1_0+o(\eps).$$
Now,
 \begin{equation}\label{diffmes}
 \begin{aligned}
 &\int_{K_2}\int_{K_1} \phi(u_1,u_2) d\mu^1_\eps d\mu^2_\eps=\int_{K_2}\int_{K_1}\phi(u_1,u_2) d\mu^1_0d\mu^2_\eps\\
 &+\eps\int_{K_2}\int_{K_1} \nabla_{u_1} \phi(u_1,u_2)\cdot d\nu^1_0d\mu^2_\eps+o(\eps)\\
 &=\int_{K_2}\int_{K_1}\phi(u_1,u_2) d\mu^1_0d\mu^2_0+\eps\int_{K_2}\nabla_{u_2}\int_{K_1}  \phi(u_1,u_2) d\mu^1_0\cdot d\nu^2_0+o(\eps)\\+&\eps\int_{K_2}\int_{K_1} \nabla_{u_1} \phi(u_1,u_2)\cdot d\nu^1_0d\mu^2_\eps +o(\eps).
  \end{aligned} 
  \end{equation}
  Using Lebesgue differentiation for the second term in \eqref{diffmes}, and $(i)$ for the third term in \eqref{diffmes}, $(ii)$ follows.
  \end{proof}
  \subsection{Linear response for operators with a spectral gap}

Let $(X,\A,m)$ be a probability space and $\B$ a Banach space continuously embedded in $L^1(X,m)$.
We assume that the constants belong to $\B$.
Let $L_\eps$ be a family of Markov operators on $\B$, for $\eps$ in a neighbourhood $V$ of 0.
We assume that for each $\eps$, $1$ is a simple eigenvalue of $L_\eps$ with an associated eigenfunction
$h_\eps$, that we normalize so that $\int h_\eps dm =1$.
\begin{proposition}\label{pro:lrmarkov}
We suppose that the operators satisfy
\begin{itemize}
\item
$\eps\mapsto L_\eps h_0 \in \B$ is differentiable at $\eps=0$
\item
$\eps\mapsto L_\eps \phi \in \B$ is continuous at $\eps=0$ for any $\phi\in \B$.
\item
$L_\eps$ has a uniform spectral gap on $\B$; i.e., $\exists \theta\in(0,1)$ and a $C>0$, independent of $\eps$, such that for any $\phi\in \B_0$ and $n\ge 1$ we have $$||L_\eps^n\phi||_{\B}\le C \theta^n ||\phi||_{\B}.$$
\end{itemize}
Under these assumptions, the eigenfunction $h_\eps$ is differentiable at $\eps=0$ as an element of $\B$.
In addition we have the linear response formula
\begin{equation}\label{eq:lrmarkov}
\partial_\eps h_\eps|_{\eps=0} = (I-L_0)^{-1} \partial_\eps L_\eps h_0 |_{\eps=0}.
\end{equation}
\end{proposition}
\begin{proof}
Since we have a uniform spectral gap the resolvent $(I-L_\eps)^{-1}$ is well defined on the set $\B_0$
of elements of $\B$ with zero average and is uniformly bounded in $\eps$.
We have the identity
\[
h_\eps = (I-L_\eps)^{-1} (L_\eps-L_0) h_0 + h_0.
\]
By assumption, setting $q=\partial_\eps L_\eps h_0$ we have 
\[
(L_\eps-L_0) h_0 = \eps q + o(\eps)
\]
where $o(\eps)$ is understood in the $\B$-norm.
Moreover, since $(L_\eps-L_0)h_0 \in\B_0$ and $\B$ is continuously embedded in $L^1$ we also have $q\in\B_0$.
Therefore we have
\[
h_\eps = \eps (I-L_\eps)^{-1} q + o(\eps).
\]
Finally, we have
\[
(I-L_\eps)^{-1}q - (I-L_0)^{-1}q = (I-L_\eps)^{-1} (L_\eps-L_0) (I-L_0)^{-1} q.
\]
The second hypothesis applied to $\phi=(I-L_0)^{-1} q$ shows that $(L_\eps-L_0) (I-L_0)^{-1} q \to 0$ as $\eps\to0$,
and the conclusion follows by uniform boundedness of $(I-L_\eps)^{-1}$ as operators on $\B_0$.
\end{proof}
\subsection{[Proof of Proposition~\ref{lem:unifgap}]
Uniqueness of stationary density and a uniform spectral gap for uniformly expanding systems}
We prove Proposition ~\ref{lem:unifgap} in a series of lemmas. Let $[\omega]_n=\omega_0, \dots, \omega_{n-1}$ be a path of length $n$ and consider the partition $\{g_{z, [\omega]_n}(X)\}_{z\in Z^n}$ of $X$, where $Z^n$ is the labelling set for the $n^{\text{th}}$-iterate of the random map. First we provide distortion estimates for  $g_{z, [\omega]_n}$. 
\begin{lemma}\label{lem:ndist}
There exists $D>0$ such that for any $n\in \NN$, $z\in Z^n$ and $x, y\in X$ we have
$$
\left|\frac{g_{z, [\omega]_n}'(x)}{g_{z, [\omega]_n}'(y)}-1\right| \le   D|x-y|.
$$
\end{lemma}

\begin{proof} Start with 
$$
\begin{aligned}
\log\left|\frac{g_{z, [\omega]_n}'(x)}{g_{z, [\omega]_n}'(y)}\right|=\sum_{j=0}^{n-1}\log \left|\frac{g_{z, \omega_j}'(g_{z, [\omega]_j}(x))}{g_{z, \omega_j}'(g_{z, [\omega]_j}'(y))}\right| \le \sum_{j=0}^{n-1} \left|\frac{g_{z, \omega_j}'(g_{z, [\omega]_j}(x))}{g_{z, \omega_j}'(g_{z, [\omega]_j}'(y))}-1\right|\\
\le \tilde D\sum_{j=0}^{n-1} |g_{z, [\omega]_j}(x)-g_{z, [\omega]_j}(y)| \le \tilde D\sum_{j=0}^{n-1}\beta^j|x-y| \le \tilde D\frac{1}{1-\beta}|x-y|.
\end{aligned}
$$
The above inequality implies that 
$$
\left|\frac{g_{z, [\omega]_n}'(x)}{g_{z, [\omega]_n}'(y)}\right|\le e^{\tilde D/(1-\beta)}.
$$
Hence, there exists $C>0$ independent of $x, y\in X$ such that  
$$
\left|\frac{g_{z, [\omega]_n}'(x)}{g_{z, [\omega]_n}'(y)}-1\right| \le C\log\left|\frac{g_{z, [\omega]_n}'(x)}{g_{z, [\omega]_n}'(y)}\right| \le C\tilde D\frac{1}{1-\beta}|x-y|.
$$
Now, we can take $D=C\tilde D/(1-\beta)$.
\end{proof}

\begin{lemma}\label{totalcontrol} 
There exists a $D>0$ such that for any $n$ we have
\begin{enumerate}
\item $(L^n_{\omega}{\bf 1}):=\sum_{z\in Z^n} |g_{z, [\omega]_n}'(x)|\le D+1;$
\item $\Lip (L^n_{\omega}{\bf 1})\le D(D+1)$;
\item $||L_{\omega}^n \Phi||_{\infty}\le (D+1) ||\Phi||_{\infty}$.
\end{enumerate}
\end{lemma}
\begin{proof}
For (1) by Lemma \ref{lem:ndist} we have $g_{z, [\omega]_n}'(x)\le (D+1)|g_{z, [\omega]_n}(X)|$. Therefore, 
$$
L_\omega^n{\bf 1} =\sum_{z\in Z^n}|g_{z, [\omega]_n}'(x)| \le D+1.
$$
Note that (3) is implied by (1). Therefore, it remains to prove (2). 
For $x, y\in X$ we have 
$$\begin{aligned}
|(L^n_{\omega}{\bf 1})(x)-(L^n_{\omega}{\bf 1})(y)| \le \sum_{z\in Z^n}|g_{z, [\omega]_n}'(x)-g_{z, [\omega]_n}'(y)| \\
\le \sum_{z\in Z^n}|g_{z, [\omega]_n}'(x)|\left|\frac{g_{z, [\omega]_n}'(y)}{g_{z, [\omega]_n}'(x)}-1 \right| \le D(D+1)|x-y|\sum_{z\in Z^n}|g_{z, [\omega]_n}(X)|.
\end{aligned}
$$
\end{proof}

\begin{lemma}\label{LYonLip}
$L_{\PP_\eps}$ admits a spectral gap on the space of Lipschitz continuous functions.
\end{lemma}
\begin{proof}
Let $[\omega]_n=\omega_0, \dots, \omega_{n-1}$ be a word of length $n$ and consider 
$$\begin{aligned}
&|L_\omega^n\Phi(x)-L_\omega^n\Phi(y)| = |\sum_{z\in Z^n} (\Phi\circ g_{z, [\omega]_n} |g_{z, [\omega]_n}'|)(x) -(\Phi\circ g_{z, [\omega]_n} |g_{z, [\omega]_n}'|)(y)| \\
&\le\sum_{z\in Z^n} \{ |\Phi(g_{z, [\omega]_n}(x))-\Phi(g_{z, [\omega]_n}(y))||g_{z, [\omega]_n}'(x)|\\
&\hskip 4cm+|\Phi(g_{z, [\omega]_n}(y))|| |g_{z, [\omega]_n}'(x)|-|g_{z, [\omega]_n}'(y)| |\}\\
&\le \Lip(\Phi)\beta^n |x-y| \sum_{z\in Z^n}|g_{z, [\omega]_n}'(x)| + ||\Phi||_\infty\Lip (L^n_{\omega}{\bf 1})|x-y| \\
&\le  (D+1)\beta^n \Lip(\Phi) |x-y|+D (D+1)||\Phi||_{\infty}|x-y|.
\end{aligned}
$$
Therefore,
$$\Lip (L^n_{\omega}\Phi)\le (D+1)\beta^n \Lip(\Phi)+D(D+1) ||\Phi||_{\infty}.$$
Now fix $n_1$ large enough so that $(D+1)\beta^{n_1}<1$. Define $\kappa:=(D+1)^{\frac{1}{n_1}}\beta<1$. We have 
$$
\Lip(L_\omega^{n_1}\Phi) \le \kappa\Lip(\Phi)+D(D+1)\|\Phi\|_\infty.
$$
Using (3) of Lemma \ref{totalcontrol}, for any $n\ge n_1$ we have
\begin{equation}\label{LipLY}
\|L_{\omega}^n\Phi\|_{\Lip}\le \kappa^n\|\Phi\|_{\Lip}+D_1\|\Phi\|_\infty,
\end{equation}
$D_1:=\frac{D(D+1)}{1-\kappa}+D(D+1).$ Consequently, for any $n\ge n_1$  
$$
\|L_{\PP_\eps}^{n}\Phi\|_{\Lip}\le \kappa^n\|\Phi\|_{\Lip}+D_1\|\Phi\|_\infty.
$$
This implies that the operators $L_{\PP_\eps}$ are quasi-compact on the space of Lipschitz continuous functions and their essential spectral radius
 is uniformly bounded by $\kappa$. Moreover, any stationary density $h_\eps\in C^{\Lip}$. It remains to show that the peripheral spectrum of $L_{\PP_\eps}$ consists only of $1$ and that $1$ is a simple eigenvalue. Indeed, if $h_\eps$ is a stationary density, then using the fact that $h_\eps$ is a $C^{\Lip}$ function, there exists an open interval $J\subset [0,1]$, such that ${h_{\eps}}_{|J}>0$. Since all the maps $T_\omega$ are piecewise onto and uniformly expanding ($\sup_{x,\omega}|g'_{z,\omega}|\le \beta<1$), there exists an $n\in\mathbb N$ such that $L^n_{\PP_\eps}({h_{\eps}}_{|J})(x)>0$. Consequently, $h_\eps(x)=L^n_{\PP_\eps}({h_{\eps}})(x)\ge L^n_{\PP_\eps}({h_{\eps}}_{|J})(x)>0$. Thus, $h_{\eps}$ is strictly positive\footnote{The argument about the strict positivity of $h_\eps$ is borrowed from \cite{SS}.} on $[0,1]$. This implies that $h_\eps$ is unique and that 1 is a simple eigenvalue of $L_{\PP_\eps}$. The fact that no other eigenvalues of $L_{\PP_\eps}$ are of modulus 1 follows from quasi-compactness of $L_{\PP_\eps}$ on $C^{\Lip}$ and repeating the above argument to show that any iterate of $L_{\PP_\eps}$ has a unique invariant density.
 \end{proof}
 \begin{lemma}\label{C1gap}\text{ }
 \begin{enumerate}
 \item[(a)] $L_{\PP_\eps}$ admit a uniform spectral gap on $C^1$. In particular their essential spectral radii on $C^1$ is uniformly bounded by above by some $\kappa\in(0,1)$.
 \item[(b)] There exists a $C>0$, such that for any $n\in \NN$ and $\Phi\in C^1$, we have 
 $$||L_{\PP_\eps}^n\Phi||_{C^1}\le C||\Phi||_{C^1}.$$
 \end{enumerate}
 \end{lemma}
 \begin{proof}
By Lemma \ref{LYonLip} the family of operators $L_{\PP_\eps}$ admit a uniform spectral gap on $C^{\Lip}$. Moreover, the operators $L_{\PP_\eps}$ preserve the space of $C^1$ functions. Indeed,
\begin{equation*}
|(L_{\PP_{\eps}}\Phi)(x)|=|\sum_{z\in Z}\int_{\Omega}\Phi\circ g_{z,\omega}(x)|g_{z,\omega}'|d\PP_\eps(\omega)|\le\tilde D||\Phi||_{C^0};
\end{equation*}
and moreover,
\begin{equation*}
\begin{split}
&|(L_{\PP_{\eps}}\Phi)'(x)|=\\
&|\sum_{z\in Z}\int_{\Omega}\Phi'\circ g_{z,\omega}(x)|g_{z,\omega}'(x)|g_{z,\omega}'+\Phi\circ g_{z,\omega} (x)\sign(g_{z,\omega}'(x)) g_{z,\omega}''(x)d\PP_\eps(\omega)|\\
&\le (L_{\PP_{\eps}}|\Phi'(x)|)+M\|\Phi\|_{C^0}.
\end{split}
\end{equation*}
Since for $C^1$ functions the Lipschitz norm is the same as the $C^1$-norm, this implies that the family of operators $L_{\PP_\eps}$ admit a uniform spectral gap on $C^{1}$. The uniform upper bound on their spectral radii follows from \eqref{LipLY}. This proves (a) of the Lemma and (b) is a consequence of (a).
 \end{proof}
Next, before proving a spectral gap of $L_{\PP_\eps}$ on $ C^2$, we prove that condition~\eqref{eq:avgdist} of (A1) can be iterated.
\begin{lemma}\label{niterate} For any $n$ and $i=2,3$ we have 
\[
\sup_{\eps\in V}\sum_{z\in Z^n}\sup_{x\in X}\int_{\Omega^n}|g^{(i)}_{z,[\omega]_n}|d\PP_\eps^{n}([\omega]_n)<\infty.
\]
\end{lemma}
\begin{proof}
We prove it by induction on $n$. For $i=2$, and $n=1$ the lemma is true by assumption. Assume it is true for $n$. Note that 
\[
(g_{z,\omega}\circ g_{z_n,\omega_n})'' = g_{z,\omega}'' \circ g_{z_n,\omega_n} \cdot g_{z_n,\omega_n}'^2 + g_{z,\omega}'\cdot g_{z_n,\omega_n}''
\]
We have for the first term 
\[
\begin{split}
\sum_{z,z_n} &\int_{\Omega^n}\int_\Omega
g_{z,\omega}'' \circ g_{z_n,\omega_n} \cdot g_{z_n,\omega_n}'^2 d\PP_\eps^{n}(\omega) d\PP_\eps(\omega_n) \\
&= \sum_{z_n}\int_\Omega g_{z_n,\omega_n}'^2 \left(\sum_{z\in Z^n}\int_{\Omega^n} g_{z,\omega}'' \circ g_{z_n,\omega_n} d\PP_\eps^{ n}(\omega)\right) d\PP_\eps(\omega_n) .
\end{split}
\]
The term inside the parenthesis is bounded by the induction hypothesis, and the remaining integral is bounded. The second term can be managed in the same way. The technique also works for the third derivative. Indeed, for $i=3$, we have
$$
\begin{aligned}
(g_{z,\omega}\circ g_{z_n,\omega_n})''' &= g_{z,\omega}''' \circ g_{z_n,\omega_n} \cdot g_{z_n,\omega_n}'^3+ 2g_{z_n,\omega_n}'\cdot g_{z_n,\omega_n}''\cdot g_{z,\omega}'' \circ g_{z_n,\omega_n}\\
&+g_{z,\omega}''\cdot g_{z_n,\omega_n}''+g_{z,\omega}'\cdot g_{z_n,\omega_n}'''.
\end{aligned}
$$
\end{proof}
\begin{remark}
The bounds in Lemma \ref{niterate} depend on $n$.
\end{remark}
\begin{lemma}
$L_{\PP_\eps}$ admits a uniform spectral gap on $C^2$. In particular, for any $\eps\in V$, the random dynamical system $(\Omega, \{T_\omega\},\PP_\eps\}$ admits a unique stationary density $h_\eps\in C^2$.
\end{lemma}
 \begin{proof}
For $\Phi\in C^2(X)$, by definition we have 
$$
L^n_{\PP_{\eps}}\Phi=\sum_{z\in Z^n}\int_{\Omega^n}\Phi\circ g_{z,[\omega_n]}|g_{z,[\omega_n]}'|d\PP^n_\eps([\omega_n]).
$$ 
By Lebesgue differentiation theorem 
\begin{equation}\label{LPPprime}
(L_{\PP_{\eps}}^n\Phi)'=\sum_{z\in Z^n}\int_{\Omega^n}\Phi'\circ g_{z,[\omega]_n}|g_{z,[\omega]_n}'|g_{z,[\omega]_n}'+\Phi\circ g_{z,[\omega]_n} \sign(g_{z,[\omega]_n}') g_{z,[\omega]_n}''d\PP^n_\eps([\omega]_n).
\end{equation}
Since $g_{z,[\omega]_n}$ is monotone for any $z\in Z^n$, again  by Lebesgue differentiation theorem for the second derivative we have 
\begin{equation}\label{LPP2prime}
\begin{aligned}
&(L_{\PP_{\eps}^n}\Phi)''=\sum_{z\in Z^n}\int_{\Omega^n}\Phi''\circ g_{z,[\omega]_n}|g_{z,[\omega]_n}'|(g_{z,[\omega]_n}')^2\\
&+ 3\sign(g_{z,[\omega]_n}')\Phi'\circ g_{z,[\omega]_n} g_{z,[\omega]_n}'g_{z,[\omega]_n}'' +\sign(g_{z,[\omega]_n}')\Phi\circ g_{z,[\omega]_n} g_{z,[\omega]_n}'''d\PP^n_\eps([\omega]_n).
\end{aligned}
\end{equation}
Using \eqref{LPPprime} \eqref{LPP2prime} we obtain  
\begin{align*}
{\|L_{\PP_{\eps}}^n\Phi\|}_{C^2}
&\le \sup_x\sum_{z\in Z^n}\int_{\Omega^n} |\Phi''\circ g_{z,[\omega]_n}||g_{z,[\omega]_n}'|^3
+ 3|\Phi'\circ g_{z,[\omega]_n}| |g_{z,[\omega]_n}'| |g_{z,[\omega]_n}''| \\
&+|\Phi\circ g_{z,[\omega]_n}| |g_{z,[\omega]_n}'''| +|\Phi'\circ g_{z,[\omega]_n}||g_{z,[\omega]_n}'|^2+|\Phi\circ g_{z,[\omega]_n}| |g_{z,[\omega]_n}''|\\
&+|\Phi\circ g_{z,[\omega]_n}||g_{z,[\omega]_n}'|
d\PP^n_\eps([\omega]_n).
\end{align*}
Therefore, using Lemma \ref{totalcontrol} and \ref{niterate} we have
\begin{align*}
{\|L^n_{\PP_{\eps}}\Phi\|}_{C^2}
&\le \beta^{2n} (D+1)\|\Phi\|_{C^2}+M_n||\Phi||_{C^1}.
\end{align*}
Fix $n_2$ large enough so that $\beta^{2n_2}(D+1)\le\gamma<1$. We have
\begin{align*}
{\|L^{n_2}_{\PP_{\eps}}\Phi\|}_{C^2}  \le \gamma|\Phi\|_{C^2}+M_{n_2}||\Phi||_{C^1}.
\end{align*}
By using, (2) of Lemma \ref{C1gap}, we iterate the above inequality (dropping $n_2$ from the notation of $L_{\PP_\eps}$ for simplicity) and get for any $n\ge 1$
\begin{align*}
{\|L^{n}_{\PP_{\eps}}\Phi\|}_{C^2}  \le \gamma^n|\Phi\|_{C^2}+\frac{C\cdot M_{n_2}}{1-\gamma}||\Phi||_{C^1}.
\end{align*}
This implies that $L_{\PP_\eps}$ is quasi-compact on $C^2$ with $\gamma$ as a uniform upper bound on the essential spectral radius. The proof of spectral gap is analogous to that in the proof of Lemma \ref{LYonLip}. 
\end{proof}
\subsection{Spectral gap for the transfer operator associated with Gauss-R\'enyi random maps}\label{SGGR}
Let \(\G(x)=\frac 1x \mod 1\) and \(\R(x)=\frac{1}{1-x}\mod 1\).
The inverse branches of \(\G\) and \(\R\) are given by 
\[
\G_n^{-1}= \frac{1}{n+x}, \R^{-1}_n(x)=1- \frac{1}{n+x} \text{ for any }  x\in\left(0, 1\right).
\]
Consequently we have 
\[
(\G_n^{-1})^\prime(x)=-\frac{1}{(n+x)^2}, (\R_n^{-1})^\prime(x)=\frac{1}{(n+x)^2}.
\]
Below we will be interested in compositions \(\G_n^{-1}\circ \G_k^{-1} \), \(\G_n^{-1}\circ \R_k^{-1} \), \(\R_n^{-1}\circ \G_k^{-1} \), \(\R_n^{-1}\circ \R_k^{-1} \). 

Thus it is useful to have exact form of the above compositions and their derivatives.
\begin{align}
&\label{GG}(\G_n^{-1}\circ \G_k^{-1})(x)=\frac{1}{n+\frac{1}{k+x}}=\frac{k+x}{n(k+x)+1},\\
&\label{RR}(\R_n^{-1}\circ \R_k^{-1})(x)=1-\frac{1}{n+1-\frac{1}{k+x}}=1-\frac{k+x}{(n+1)(k+x)-1}.
\end{align}
\begin{remark}\label{Rem:1}
Notice that 
\[
(\R_n^{-1}\circ \G_k^{-1})(x)=1-(\G_n^{-1}\circ \G_k^{-1})(x), (\R_n^{-1}\circ \R_k^{-1})(x)=1-(\G_n^{-1}\circ \R_k^{-1})(x).
\]
Consequently, 
\[
(\R_n^{-1}\circ \G_k^{-1})'(x)=-(\G_n^{-1}\circ \G_k^{-1})'(x), (\R_n^{-1}\circ \R_k^{-1})'(x)=-(\G_n^{-1}\circ \R_k^{-1})(x).
\]
\end{remark}
In the light of remark \ref{Rem:1} it is sufficient to consider the two cases \(\G\circ \G\) and \(\R\circ\R\). The cases \(\R\circ \G\) and \(\G\circ\R\) then satisfy the same type of estimates.   By the chain rule we have 
\begin{align}
&\label{GG'}(\G_n^{-1}\circ \G_k^{-1})'(x)=\frac{1}{\left(n+\frac{1}{k+x}\right)^2}\frac{1}{(k+x)^2}=\frac{1}{(n(k+x)+1)^2},\\
&\label{RR'}(\R_n^{-1}\circ \R_k^{-1})'(x)=\frac{1}{\left(n+1-\frac{1}{k+x}\right)^2}\frac{1}{(k+x)^2}=\frac{1}{((n+1)(k+x)-1)^2},\\
&\label{GG''}(\G_n^{-1}\circ \G_k^{-1})''(x)=-\frac{2n}{(n(k+x)+1)^3}, \\
&\label{RR''}(\R_n^{-1}\circ \R_k^{-1})''(x)=-\frac{2(n+1)}{((n+1)(k+x)-1)^3}\\
&\label{GG'''}(\G_n^{-1}\circ \G_k^{-1})'''(x)=\frac{6n^2}{(n(k+x)+1)^4}\\ 
&\label{RR'''}(\R_n^{-1}\circ \R_k^{-1})'''(x)=-\frac{6(n+1)^2}{((n+1)(k+x)-1)^4}
\end{align}
The following lemma is a straightforward consequence of \eqref{GG'} - \eqref{RR''}.
\begin{lemma}\label{01}Let \(g_{n, k}\in\{\Phi\circ \Psi\mid \Phi, \Psi\in\{\G_n^{-1}, \R_k^{-1}\}\}\). Then  for any \(n, k\in \NN\) and \(x\in (0, 1)\) hold
\begin{equation}\label{gnk'}
|g_{n, k}^{(i)}(x)|\le \frac{i!}{n^2k^{i+1}}, \quad \text{for \(i=1, 2, 3\)}.
\end{equation}
Moreover, if  \(g_{n, k}=\G_n^{-1}\circ \G_n^{-1}\) then
\begin{equation}\label{GG'nk}
|g_{n, k}'(x)|\le \frac{1}{4} \quad \text{for any \(n, k\in \NN\) and \(x\in (0, 1)\).}
\end{equation}
\end{lemma}

Recall that the transfer operator associated to the random Gauss-R\'eyni map has the following form
\[
L_{\PP_\eps}f=(1-\eps)^2L_{\G\circ\G}f+\eps(1-\eps)L_{\G\circ\R}f+\eps(1-\eps)L_{\R\circ\G}f+\eps^2 L_{\R\circ\R}f.
\]
\begin{lemma}
\(L_{\PP_\eps}\) admits uniform spectral gap in \(C^i\), \(i=1, 2\).
\end{lemma}
\begin{proof}
The proof is  straightforward computation and uses the fact that in our situation we can differentiate term by term in the series. For any  \(\Phi, \Psi\in\{\G_n, \R_k\}\) by definition and  \eqref{gnk'} we have 
\begin{equation}\label{LFiPsif}
{\|L_{\Phi\circ \Psi}f\|}_{C^0}\le {\|f\|}_{C^0}\sum_{n, k=1}^\infty{\|g_{n, k}'\|}_{C^0}\le {\|f\|}_{C^0} \sum_{n, k=1}^\infty\frac{1}{n^2k^2}\le \frac{\pi^4}{36}{\|f\|}_{C^0}.
\end{equation}
Below we will use the fact that derivatives of \(g_{n, k}\) do not change their sign. Thus we ignore the absolute values while taking derivatives from the expressions of form \(|g_{n, k}'|^i\). 
\begin{equation}\label{LFiPsif'}
\begin{aligned}
&{\|(L_{\Phi\circ \Psi}f)'\|}_{C^0}\le {\|f'\|}_{C^0}\sum_{n, k=1}^\infty{\|(g_{n, k}')^2\|}_{C^0}+{\|f\|}_{C^0}\sum_{n, k=1}^\infty{\|g_{n, k}''\|}_{C^0}\\
&\le   {\|g_{n, k}'\|}_{C^0}{\|f'\|}_{C^0}\sum_{n, k=1}^\infty\frac{1}{n^2k^2}+2{\|f\|}_{C^0}\sum_{n, k=1}^\infty\frac{1}{n^2k^3}\le 
 {\|g_{n, k}'\|}_{C^0}\frac{\pi^4}{36} {\|f'\|}_{C^0}+\frac{\pi^2}{2}{\|f\|}_{C^0},
\end{aligned}
\end{equation}
where we used Lemma \ref{01} and the fact \(\zeta(3)=\sum_k k^{-3}\le 3/2\) in the last two chain of inequalities.
Finally,  again by triangle inequality and Lemma \ref{01} we have 
\begin{equation*}
\begin{aligned}
&{\|(L_{\Phi\circ \Psi}f)''\|}_{C^0}\le {\|f''\|}_{C^0}\sum_{n, k=1}^\infty{\|g_{n, k}'\|}_{C^0}^3+3{\|f'\|}_{C^0}\sum_{n, k=1}^\infty{\|g_{n, k}'g_{n, k}''\|}_{C^0}\\
&+{\|f\|}_{C^0}\sum_{n, k=1}^\infty{\|g_{n, k}'''\|}_{C^0} \le \frac{\pi^4}{36} {\|g_{n, k}'\|}_{C^0}^2 {\|f''\|}_{C^0} +6\zeta(4)\zeta(5){\|f'\|}_{C^0} +6\zeta(2)\zeta(4){\|f\|}_{C^0},
\end{aligned}
\end{equation*}
which implies
\begin{equation} \label{LFiPsif''}
{\|(L_{\Phi\circ \Psi}f)''\|}_{C^0}\le \frac{\pi^4}{36} {\|g_{n, k}'\|}_{C^0}^2 {\|f''\|}_{C^0} +\frac{11\pi^2}{150}{\|f'\|}_{C^0} +\frac{\pi^6}{90}{\|f\|}_{C^0},
\end{equation}

We now prove a uniform Lasota-Yorke type inequality for \(L_{\Phi\circ\Psi}\) acting on \(C^2\) for all possible choices \(\Phi, \Psi\in\{\G_n, \R_k\}\). 

First consider the case \(g_{k, n}=\G^{-1}_n\circ \G^{-1}_k\). Notice that \(g_{n, k}'\le {1}/{4}\). The equations \eqref{LFiPsif} and \eqref{LFiPsif'} immediately  imply that 
\begin{equation*}
{\|L_{\G\circ\G}f\|}_{C^2}\le \frac{1}{16}\frac{\pi^4}{36}{\|f''\|}_{C^0} + \left(\frac{11\pi^2}{150}+\frac{\pi^4}{36} \right){\|f'\|}_{C^0} +\left( \frac{\pi^6}{90} +\frac{\pi^2}{2} +\frac{\pi^4}{36}\right){\|f\|}_{C^0}
\end{equation*}
This immediately implies that 
\begin{equation}\label{LGG}
{\|L_{\G\circ\G}f\|}_{C^2}\le \frac{1}{16}\frac{\pi^4}{36}{\|f''\|}_{C^0} + \frac{158}{45}{\|f'\|}_{C^0} +\frac{170}{9}{\|f\|}_{C^0}.
\end{equation}
 Also, in the light of Remark \ref{Rem:1} we have 
 \begin{equation}\label{LRG}
{\|L_{\G\circ\R}f\|}_{C^2}\le \frac{1}{16}\frac{\pi^4}{36}{\|f''\|}_{C^0} + \frac{158}{45}{\|f'\|}_{C^0} +\frac{170}{9}{\|f\|}_{C^0}.
\end{equation}

 When \(g_{k, n}=\R^{-1}_n\circ \R^{-1}_k\) we only have  \(g_{n, k}'\le 1\). Thus  the equations \eqref{LFiPsif} and \eqref{LFiPsif'}   imply that 
\begin{equation}\label{LRR}
{\|L_{\R\circ\R}f\|}_{C^2}\le\frac{\pi^4}{36}{\|f''\|}_{C^0} + \frac{158}{45}{\|f'\|}_{C^0} +\frac{170}{9}{\|f\|}_{C^0}.
\end{equation}
Again, by Remark \ref{Rem:1}  we have 

\begin{equation}\label{LGR}
{\|L_{\R\circ\R}f\|}_{C^2}\le\frac{\pi^4}{36}{\|f''\|}_{C^0} + \frac{158}{45}{\|f'\|}_{C^0} +\frac{170}{9}{\|f\|}_{C^0}.
\end{equation}

Finally for \(L_{\PP_\eps}\)  we have 
\begin{equation}\label{LPPeps}
\begin{aligned}
{\|L_{\PP_\eps}f''\|}_{C^0} \le & (1-\eps)^2{\|L_{\G\circ \G}f''\|}_{C^0}
 +\eps(1-\eps){\|L_{\G\circ \R}f''\|}_{C^0}
\\& +\eps(1-\eps){\|L_{\R\circ \G}f''\|}_{C^0}
 +\eps^2{\|L_{\R\circ \R}f''\|}_{C^0}.
\end{aligned}
\end{equation}
Substituting the equations \eqref{LGG}-\eqref{LRR} into \eqref{LPPeps} implies that 
\[
{\|L_{\PP_\eps}f''\|}_{C^0} \le \frac{\pi^4(1+17\eps)}{576}{\|f''\|}_{C^0} +\frac{158}{45}{\|f'\|}_{C^0} +\frac{170}{9}{\|f\|}_{C^0}.
\]
Now we choose \(\eps>0\) small enough and \(M>0\)  so that \(\kappa:=\dfrac{\pi^4(1+17\eps)}{576}<1\)  and 
\[
{\|L_{\PP_\eps}f''\|}_{C^0} \le \kappa{\|f''\|}_{C^0} +M{\|f\|}_{C^1}.
\]
Thus \(L_{\PP_\eps}\) is quasi compact on \(C^2\) and the essential spectral radius is at most  \(\kappa\). The proof of spectral gap on $C^2$ is analogous to that in the proof of Lemma \ref{LYonLip}. In a similar manner, using \eqref{LFiPsif} and \eqref{LFiPsif'}, one can obtain a uniform Lasota-Yorke inequality and a uniform spectral gap on $C^1$. 
\end{proof}

\end{document}